\documentclass[reqno,11pt]{amsart} 
\usepackage{esint,mathrsfs,amsmath,amssymb,amsfonts,bbm}
\usepackage{verbatim}

\usepackage{color}
\usepackage[hyperfootnotes=false]{hyperref}
\hypersetup{
  colorlinks,
  citecolor=blue,
  linkcolor=red,
  urlcolor=blue}

\usepackage[left=1.2in, right=1.2in, bottom=1.5in]{geometry}
\newtheorem{theorem}{Theorem}[section]
\newtheorem{lemma}[theorem]{Lemma}

\theoremstyle{definition}

\theoremstyle{remark}
\newtheorem{remark}[theorem]{Remark}

\numberwithin{equation}{section}

\newcommand{\cL}{\mathcal{L}}

\newcommand{\Z}{\mathbb{Z}}

\newcommand{\R}{\mathbb{R}}

\newcommand{\e}{\varepsilon}

\newcommand{\dist}{\text{dist}}

\newcommand{\jz}[1]{\textcolor{red}{#1}}

\begin{document}

\title[Lipschitz and $W^{1,p}$ estimates]{Uniform regularity for degenerate elliptic equations in  perforated domains}


\author{Zhongwei Shen}
\address{Zhongwei Shen: Department of Mathematics, University of Kentucky, Lexington, Kentucky 40506, USA.}
\email{zshen2@uky.edu}

\author{Jinping Zhuge}
\address{Jinping Zhuge: Morningside Center of Mathematics, Academy of Mathematics and Systems Science,
Chinese Academy of Sciences, Beijing, China.}
\email{jpzhuge@amss.ac.cn}



\maketitle

\begin{abstract}

This paper is concerned with a class of degenerate elliptic equations with rapidly oscillating coefficients in  periodically perforated domains, which arises in the study of spectrum problems for uniformly elliptic equations in perforated domains. We establish a quantitative convergence rate and obtain the uniform weighted Lipschitz and $W^{1,p}$ estimates.

\end{abstract}

\section{Introduction}

In this paper we consider a class of degenerate elliptic operators with rapidly oscillating coefficients,
\begin{equation}\label{op}
\mathcal{L}_\e =-\text{\rm div} (\phi_\e^2(x)  A_\e (x)  \nabla ),
\end{equation}
in a periodically perforated domain $\Omega_\e$, where $\phi_\e (x)=\phi(x/\e)$ is a periodic scalar function
that vanishes on the boundaries of holes and $A_\e (x) =A(x/\e)$
a periodic  matrix-valued function. The parameter $\e>0$ is assumed
to be small. The operator $\cL_\e$ arises in the study of the asymptotic expansions of the spectrum
for the uniformly elliptic operator $-\text{\rm div} (A_\e(x)  \nabla )$ in perforated domains; see \cite{V81} and \cite[Chapter III]{OSY92}.

To describe $\Omega_\e$, let $Y=(-1/2, 1/2)^d$ be a unit cell in $\mathbb{R}^d$ and $\{ \tau_i: i=1, 2, \dots, m \}$ a finite
number of mutually disjoint open sets (holes) in $Y$ with smooth and connected boundaries.
We assume that dist$(\tau_i, \partial Y)\ge c_0$ and dist$(\tau_i, \tau_j)\ge c_0$ for $i\neq j$, where $c_0>0$.
Let 
\begin{equation*}
T=\bigcup_{i=1}^m \overline{\tau_i} \quad \text{ and } \quad
T_\e=\bigcup_{z\in \mathbb{Z}^d} \e (z+ T).
\end{equation*}
For a bounded domain $\Omega$ in $\mathbb{R}^d$ and $0< \e< 1$, we define
\begin{equation*}
\Omega_\e=\Omega\setminus T_\e=\Omega \setminus \bigcup_{z\in \mathbb{Z}^d} \e (z+T).
\end{equation*}
Throughout the paper we shall assume that 
\begin{equation}\label{H}
\text{\rm dist} (\partial \Omega, \partial T_\e) \ge c_0
\end{equation}
for some $c_0>0$. As a result, $\partial \Omega_\e =\partial \Omega \cup \Gamma_\e$,
where $\Gamma_\e =\Omega \cap \partial T_\e $ and dist$(\partial \Omega, \Gamma_\e)\ge c_0$. The restrictive geometric assumption \eqref{H} is essential for the uniform regularity near the boundary $\partial\Omega$. The typical examples in applications satisfying this assumption include rectangular domains with sides parallel to the coordinate planes.

Next, we describe the conditions on $\phi$ and $A$.
Throughout the paper, we assume that $\phi=\phi(y)$ and $A=A(y)$ are $Y$-periodic, and that the matrix $A=A(y)=(a_{ij}(y))_{d\times d} $ is symmetric, H\"older continuous, and uniformly elliptic, i.e.,
\begin{equation}
    \langle A \xi, \xi\rangle \ge \mu |\xi|^2 \quad \text{ for any } \xi \in \mathbb{R}^d,
\end{equation}
for some $\mu>0$.
We assume that $\phi(x)$ degenerates as a distance function from the holes, i.e.,
\begin{equation}\label{cond.phi-1}
\phi(y) \approx \text{\rm dist}(y, T) \quad \text{ for } y \in Y,
\end{equation}
where the notation $a\approx b$ indicates that $a$ is comparable to $b$ in the sense that $cb \le a\le C b$ for some universal constants $c,C>0$. It is worth noting that $\phi^2$ fails to be a Muckenhoupt $A_2$ weight. We refer the reader  to \cite{FKS82,WWYZ07,STV21,DP21, DP20} and references therein for
regularity estimates of degenerate or singular elliptic equations.
 We further assume that  $\phi\in C^{1, \alpha} (Y\setminus T)$ for some $\alpha>0$ and that 
\begin{equation}\label{cond.phi-2}
\|   \text{\rm div} (A \nabla \phi ) \|_{L^{p_0}(Y\setminus T)} \le C
\end{equation}
for some $p_0>d$ (this is needed for the small-scale Lipschitz estimate). Clearly, this condition is satisfied if $A$ and $\phi$ are sufficiently smooth. Another example is the case that $\phi$ is 
the ground state for the Schr\"{o}dinger operator $-\text{\rm div}(A\nabla) +V$ in $Y\setminus T$ with 
periodic conditions on $\partial Y$ and Dirichlet condition on $\partial T$,
where $V \in  L^{p_0} (Y\setminus T)$ for some $p_0> d$. In the application to the spectral problems (see \cite{V81,OSY92} and \eqref{e-a} below), $\phi$ is the principal Dirichlet eigenfunction of $-\text{div}(A\nabla)$,
which satisfies  both \eqref{cond.phi-1} and \eqref{cond.phi-2}.

The following two theorems are the main results of this paper.

\begin{theorem}\label{main-thm-1}
Suppose that $u_\e$ is the weak solution of
\begin{equation}\label{main.eq.phief}
\cL_\e( u_\e) = \phi_\e f \quad \text{ in } \Omega_\e\quad \text{ and }
\quad u_\e =0 \quad \text{ on } \partial \Omega.
\end{equation}
\begin{itemize}
    \item[(i)] If $\Omega$ is a $C^{1,\alpha}$ domain and $p>d$, then
\begin{equation}\label{est.GlobalLip}
    \| \phi_\e \nabla u_\e \|_{L^\infty(\Omega_\e)}
  \le C_p \| f\|_{L^p(\Omega_\e)}.
 \end{equation}

\item[(ii)] If $\Omega$ is $C^1$ or convex and $1< p< d $, $ \frac{1}{p^*}=\frac{1}{p}-\frac{1}{d}$, then
\begin{equation}
    \| \phi_\e \nabla u_\e \|_{L^{p^*}(\Omega_\e)}   \le C_p \| f \|_{L^p(\Omega_\e)}.
\end{equation}
\end{itemize}
In both cases, the constant $C_p$ is independent of $\e \in (0, 1)$.
\end{theorem}

\begin{theorem}\label{main-thm-2}
Suppose that $\Omega$ is a $C^1$ or convex domain and $u_\e$  the weak solution of 
\begin{equation}\label{main.eq.F}
\cL_\e( u_\e) = \text{\rm div} (\phi_\e f) +F \quad \text{ in } \Omega_\e\quad \text{ and }
\quad u_\e =0 \quad \text{ on } \partial \Omega.
\end{equation}
Then for $1< p< \infty$,
\begin{equation}
\| \phi_\e \nabla u_\e \|_p 
\le C_p \left\{ \| f\|_p + \| F\|_p \right\},
\end{equation}
where $C_p$ is independent of $\e \in (0, 1)$.    
\end{theorem}

Thanks to the Lax-Milgram Theorem, the boundary value problems
\eqref{main.eq.phief} and \eqref{main.eq.F} are solvable in the energy space $V_\e$; see Section \ref{sec.3}. We 
point out that no boundary condition is needed on $\Gamma_\e=\partial \Omega_\e \setminus \partial \Omega$  due to the strong 
degeneracy of the coefficients near  $\Gamma_\e$. The regularity results in the main theorems can be applied to the Dirichlet eigenfunctions of the operator $-\text{div}(A_\e \nabla)$ in the perforated domain $\Omega_\e$; see \cite{OSY92}.
Indeed, a computation shows that if $-\text{\rm div} (A_\e \nabla u_\e) =\lambda_\e u_\e$ in $\Omega_\e$ and
$u_\e =\phi_\e v_\e$, where $\phi$ is the principal eigenfunction of $-\text{\rm div}(A\nabla) $ in $Y\setminus T$ with 
the principal eigenvalue $\overline{\lambda}$, then
\begin{equation}\label{e-a}
-\text{\rm div} (\phi_\e^2 A_\e\nabla v_\e )  = (\lambda_\e -\e^{-2} \overline{\lambda})  \phi_\e^2 v_\e \quad \text{ in } \Omega_\e.
\end{equation}
In \cite{V81} the observation above was used to show that $\lambda^k_\e = \e^{-2} \overline{\lambda} + \mu^k_0 + o(1) $ as $ \e\ \to 0$,
where $\lambda_\e^k$ is the $k$th Dirichlet eigenvalue of $-\Delta $ in $\Omega_\e$ and
 $\mu^k_0$ the $k$th Dirichlet  eigenvalue for a second-order elliptic operator  with constant coefficients in $\Omega$.
Using some uniform estimates for eigenfunctions, it was proved in \cite{OSY92} that 
$$
|\lambda^k_\e -\e^{-2} \overline{\lambda} - \mu^k_0| \le C \e,
$$
for sufficiently small $\e$.

The proofs of Theorem \ref{main-thm-1} and Theorem \ref{main-thm-2} boil down into estimates at two different scales. At small scales (below $\e$), we need to establish the weighted Lipschitz and $W^{1,p}$ estimates in a cell with $\e=1$. Due to the strong degeneracy of the coefficients (see \eqref{cond.phi-1}), the classical Schauder estimates as well as those for degenerate equations in the existing literature 
cannot apply directly. Instead, we use a technique of Moser's iteration to establish an unweighted $L^\infty$ estimate
for the equation $-\text{\rm div}(\phi^2 A\nabla u) =\phi F$.
A transformation, which reduces the degenerate equation to a nondegenerate equation,
is then used to prove the weighted Lipschitz estimates. This gives us the small-scale Lipschitz estimates and therefore the small-scale $W^{1,p}$ estimates by a real-variable argument. 
At large scales, we first establish a quantitative convergence rate and use it to obtain the uniform estimates above the $\e$-scale. A careful analysis involving the harmonic extension and weighted Sobolev inequalities
is carried out to handle the difficulties caused by the degeneracy of the coefficients and the holes. 
Finally, the excess decay iteration  and the real-variable argument, as standard tools in the homogenization theory, are
used to establish the weighted Lipschitz estimate and $W^{1,p}$ estimate, respectively.
For large-scale regularity estimates for uniformly elliptic operators with oscillating  coefficients,
we refer to the reader to \cite{A-Lin87,KLS2013, AS16,A-Shen16,GNO20} and references therein.

The rest of the paper is organized as follows. In Section \ref{sec.2}, we introduce the weighted Sobolev spaces and establish some useful inequalities of independent interest. In Section \ref{sec.3}, we establish the Lipschitz estimate at small scales for the degenerate elliptic equations. In Section \ref{sec.4}, we use the techniques of homogenization to derive a quantitative convergence rate and a local first-order approximation. In Section \ref{sec.5}, we establish the large-scale Lipschitz estimate and prove Theorem \ref{main-thm-1} part (i). In Section \ref{sec.6}, we establish both the small-scale and large-scale $W^{1,p}$ estimates and prove Theorem \ref{main-thm-2} and Theorem \ref{main-thm-1} part (ii). In Appendix, we list some properties of the smoothing operators and nontangential maximal functions that have been used in this paper.

\noindent\textbf{Acknowledgements.} Z.S. is partially supported by the NSF grant DMS-2153585.  J.Z. is partially supported by grants for Excellent Youth from the NSFC and AMSS-CAS. The authors would like to thank Professors  Hongjie Dong and Tuoc Phan for pointing out the related work in \cite{DP21, DP20}.

\section{Inequalities in weighted Sobolev spaces}\label{sec.2}

 In this section, we establish several useful inequalities in $\phi$- or $\phi_\e$-weighted Sobolev spaces in the $L^p$ settings. 
 Some of the inequalities in the case $p=2$ may be found in \cite{V81, OSY92}. 

Recall that $Y_*=(-1/2, 1/2)^d\setminus T$.
For $1\le p< \infty$, let
$$
L^p_\phi(Y_*) =\left \{ v\in L^1_{\rm loc}(Y_*): v\phi \in L^p(Y_*) \right\},
$$
$$
L^p_{\phi_\e}(\Omega_\e) = \left\{ v\in L^1_{\rm loc}(\Omega_\e): v\phi_\e \in L^p(\Omega_\e) \right\}.
$$
Let $W^{1,p}_{\phi}(Y_*)$ be a $\phi$-weighted Sobolev space defined by
\begin{equation*}
    W^{1,p}_{\phi}(Y_*) = \left\{ v\in L^1_{\rm loc}(Y_*): v \in L^p_\phi(Y_*) \text{ and }  \nabla v \in L^p_\phi(Y_*)^d \right\}.
\end{equation*}
Similarly, define the $\phi_\e$-weighted Sobolev space in $\Omega_\e$ by
\begin{equation*}
    W^{1,p}_{\phi_\e}(\Omega_\e) = \left\{ v\in L^1_{\rm loc}(\Omega_\e):  v \in L^p_{\phi_\e}(\Omega_\e) \text{ and } 
    \nabla v \in L^p_{\phi_\e}(\Omega_\e)^d\right\}.
\end{equation*}
Let $W^{1,p}_{\phi_\e,0}(\Omega_\e)$ denote the subspace of $W^{1,p}_{\phi_\e}(\Omega_\e)$ that contains the functions vanishing on $\partial \Omega$.


A key technique for studying the weighted Sobolev spaces is the harmonic extension. For each hole $\tau_i \subset T$, let $\tau'_i$ be the extended smooth hole such that $\tau_i \subset \tau_i'$ and 
$$
\dist(\tau_i',\tau_j') \ge c_0/4, \quad  \dist(\tau_i',\partial Y) \ge c_0/4, \quad \dist(\partial \tau_i', \tau_i) \ge c_0/4.
$$
Let $T' = \cup_{i=1}^m \overline{\tau_i'} \subset Y$ be the union of the extended holes and  $Y_*' = (-1/2, 1/2)^d\setminus T'$.
The harmonic extension operator 
$$
\mathcal{E}: W^{1,p}(Y_*') \to W^{1,p}(Y)
$$
is a linear operator defined as follows: for every $f\in W^{1,p}(Y'_*)$, $\mathcal{E} (f) = f$ in $Y'_*$,
$\Delta \mathcal{E} (f) = 0$ in $T'$ and $\mathcal{E} (f) = f$ on $\partial T'$ in the sense of trace. In this case, we say $\mathcal{E}(f)$ is the harmonic extension of $f$ from $Y'_*$ to $Y$.

\begin{lemma}\label{lem.Gradient Ext}
    For each $1< p<\infty$, the operator $\mathcal{E}: W^{1,p}(Y_*') \to W^{1,p}(Y)$ is bounded. Moreover,
    \begin{equation*}
        \| \nabla \mathcal{E}(f) \|_{L^p(Y)} \le C \| \nabla f\|_{L^p(Y_*')}.
    \end{equation*}
\end{lemma}

\begin{proof}
Let $f\in W^{1, p}(Y^\prime_*)$.
    Since $T$ is $C^1$, the classical extension theorem in Sobolev spaces implies that there exists $\tilde{f} \in W^{1,p}(Y)$ such that $\tilde{f} = f$ in $Y'_*$. Moreover, $\| \tilde{f} \|_{W^{1,p}(Y)} \le C\| f \|_{W^{1,p}(Y'_*)}$.
    Let $u_f = \mathcal{E}(f)$ be the solution of the 
    Dirichlet problem $-\Delta u_f = 0$ in $T'$ and $u_f = f$ on $\partial T'$ in the sense of trace. Note that $v_f:= u_f - \tilde{f}$ satisfies $-\Delta v_f = \nabla\cdot \nabla \tilde{f}$ in $T'$ and $v_f = 0$ on $\partial T'$. Now the $W^{1,p}$ estimate leads to 
    $$\| \nabla v_f \|_{L^p(T')} \le C \| \nabla \tilde{f} \|_{L^p(T')} \le C\| f \|_{W^{1,p}(Y'_*)}.
    $$
    By the triangle inequality,
    \begin{equation}\label{est.uf}
        \| \nabla u_f \|_{L^p(T')} \le \| \nabla v_f \|_{L^p(T')} + \| \nabla \tilde{f} \|_{L^p(T')} \le C\| f \|_{W^{1,p}(Y'_*)}.
    \end{equation}
    Note that $u_f$ is the harmonic extension of $f$ from $Y'_*$ to $Y$. Then for any constant $L$, $u_f - L$ is the harmonic extension of $f-L$. Let $L = \fint_{Y'_*} f$ and apply \eqref{est.uf} to $f-L$, we have
    \begin{equation*}
        \| \nabla u_f \|_{L^p(T')} \le C\| f -L \|_{W^{1,p}(Y'_*)} \le C\| \nabla f \|_{L^p(Y'_*)},
    \end{equation*}
    where we have used the fact that $Y_*^\prime$ is connected.
    Since $\mathcal{E}(f) = f$ in $Y'_*$ and $\mathcal{E}(f) = u_f$ in $T'$, the last inequality yields the desired estimate.
\end{proof}


The next two lemmas give Sobolev embedding theorems in the weighted spaces $W^{1,p}_{\phi}(Y_*)$. 


\begin{lemma}\label{lem.local Poincare}
    Let $1\le p<\infty$.  Then there exists a constant $C$ such that
    \begin{equation*}
        \| u \|_{L^p(Y_*)} \le C\| \phi \nabla u \|_{L^p(Y_*)} + C\| u \|_{L^p(Y'_*)}
    \end{equation*}
    for any $u\in W^{1,p}_{\phi}(Y_*)$.
\end{lemma}

\begin{proof}
    Since $\phi(x) \approx \dist(x,T)$,
    the inequality is essentially a version of Hardy's inequality. We give an elementary proof. By flatting the boundary of each hole locally, it suffices to show a one-dimensional estimate:
    \begin{equation}\label{est.1DHardy}
        \int_0^\frac12 |v(t)|^p dt \le C\int_0^1 |t v'(t)|^p dt + C\int_{\frac12}^1 |v(t)|^p dt.
    \end{equation}
    Let $\eta$ be a smooth cutoff function such that $\eta = 1$ in $[0,1/2]$,  $\eta(1) = 0$ and $|\eta'| \le 3$. Then for $t\in (0,1/2)$,
    \begin{equation*}
        |v(t)|^p = \bigg|\int_t^1 (|v(s)|^p \eta(s))'ds\bigg| \le C \int_t^1 |v(s)|^{p-1}|v'(s)| ds + C\int_{\frac12}^1 |v(s)|^p ds.
    \end{equation*}
    Thus, by the Fubini theorem and the Young inequality,
    \begin{equation*}
    \begin{aligned}
        \int_0^\frac12 |v(t)|^p dt & \le C\int_0^{1}  |v(s)|^{p-1} s|v'(s)| ds + C\int_{\frac12}^1 |v(s)|^p ds \\
        & \le \frac12 \int_0^1 |v(s)|^p ds + C\int_0^1 |sv'(s)|^p ds + C\int_{\frac12}^1 |v(s)|^p ds.
    \end{aligned}
    \end{equation*}
    This gives the desired inequality \eqref{est.1DHardy}.
\end{proof}

For convenience, for $p\in [1,d)$, let $p^*$ be given by $\frac{1}{p^*} = \frac{1}{p} - \frac{1}{d}$.

\begin{lemma}\label{lem.localSobolev}
    Let $1\le p<\infty$. Assume $u\in W^{1,p}_{\phi}(Y_*)$. 
    \begin{itemize}
        \item[(i)] For $1\le p<d$, there exists a constant $C$ such that
    \begin{equation}\label{est.local Sobolev}
        \| \phi u \|_{L^{p^*}(Y_*)} \le C\| \phi \nabla u \|_{L^p(Y_*)} + C\| u \|_{L^p( Y'_*)}.
    \end{equation}

     \item[(ii)] For $p>d$, there exists a constant $C$ such that
    \begin{equation}\label{est.Ca.phiu}
        \| \phi u \|_{C^\alpha(Y_*)} \le C\| \phi \nabla u \|_{L^p(Y_*)} + C\| u \|_{L^p(Y'_*)},
    \end{equation}
    where $\alpha = 1-\frac{d}{p}$.
    \end{itemize}
\end{lemma}

\begin{proof}
Let $1\le  p< d$.
    By Lemma \ref{lem.local Poincare}, $u\in L^p(Y_*)$, which, together with the fact $\phi(x)\approx \text{\rm dist}(x, T)$, implies that
    $\phi u = 0$ on $\partial T$. 
    We apply the Sobolev-Poincar\'{e} inequality in $Y_*$ to obtain
    \begin{equation*}
    \begin{aligned}
        \| \phi u \|_{L^{p^*} (Y_*)} & \le C\| \nabla(\phi u ) \|_{L^p(Y_*)} \\
        & \le C\| \phi \nabla u \|_{L^p(Y_*)} + C\| u \|_{L^p(Y_*)} \\
        & \le C\| \phi \nabla u \|_{L^p(Y_*)} + C\| u \|_{L^p(Y'_*)},
    \end{aligned}
    \end{equation*}
    where we have used Lemma \ref{lem.local Poincare} and the assumption $\|\nabla \phi\|_{L^\infty(Y_*)} \le C$.
     This proves (i). The proof of (ii) is similar.
\end{proof}

Now we establish several useful weighted inequalities in the perforated domain $\Omega_\e$. The main idea is a technique 
of ``multiscale inequalities'', which means that a general inequality over $\Omega_\e$ is broken down into inequalities in each $\e$-periodic cells, together with a global estimate in $\Omega$ with no holes.
Let
\begin{equation*}
     T'_\e = \bigcup_{z\in \Z^d} \e(T'+z) \quad \text{ and } \quad   \Omega'_\e = \Omega \setminus T'_\e.
\end{equation*}
Without loss of generality, we assume that the extended holes $\e(z+T')$ do not intersect with the boundary $\partial \Omega$.
This allows us to define the harmonic extension for any function $u \in W^{1,p}(\Omega'_\e)$. Precisely, let $\mathcal{E}_\e: W^{1,p}(\Omega'_\e) \to W^{1,p}(\Omega)$ be the harmonic extension operator such that $\mathcal{E}_\e(u) = u$ in $\Omega'_\e$ and $\Delta \mathcal{E}_\e(u) = 0$ in $T'_\e \cap \Omega$.



\begin{theorem}
[Embedding from weighted to unweighted]\label{thm.unweigthed}
    Let $1< p<\infty$. There exists $C>0$ such that for any $u\in W^{1,p}_{\phi_\e,0}(\Omega_\e)$,
    \begin{equation*}
        \| u \|_{L^p(\Omega_\e)} \le C\| \phi_\e \nabla u \|_{L^p(\Omega_\e)}.
    \end{equation*}
\end{theorem}

\begin{proof}
    We apply a technique of ``multiscale inequalities''. First, we see that $u\in W^{1,p}_{\phi_\e,0}(\Omega_\e)$ implies $u\in W^{1,p}(\Omega'_\e)$. Let $\tilde{u} = \mathcal{E}_\e(u)$ be the harmonic extension of $u$. Lemma \ref{lem.Gradient Ext} implies
    \begin{equation*}
        \int_{\e(z+Y) \cap \Omega} |\nabla \tilde{u}|^p \le C\int_{\e(z+Y'_*) \cap \Omega} |\nabla u|^p \le C\int_{\e(z+Y'_*) \cap \Omega} \phi_\e^p|\nabla u|^p.
    \end{equation*}
    Summing over $z$ (including the boundary cells), we obtain
    \begin{equation}\label{est.extu}
        \int_{\Omega} |\nabla \tilde{u}|^p \le C\int_{\Omega_\e} \phi_\e^p|\nabla u|^p.
    \end{equation}
    Note that $\partial\Omega$ does not intersect with $\e(z+T'_*)$. Thus,  $\tilde{u} = u = 0$ on $\partial \Omega$. The Poincar\'{e} inequality gives
    \begin{equation}\label{est.up2phiDup}
        \int_{\Omega'_\e} |u|^p \le \int_{\Omega} |\tilde{u}|^p \le C\int_{\Omega} |\nabla \tilde{u}|^p\le C\int_{\Omega_\e} \phi_\e^p|\nabla u|^p.
    \end{equation}

    Next, we consider each cell $\e(z+Y_*)$  in $\Omega_\e$. By Lemma \ref{lem.local Poincare} and rescaling,
    \begin{equation*}
        \int_{\e(z+Y_*) \cap \Omega} |u|^p \le C\e^p \int_{\e(z+Y_*) \cap \Omega} \phi_\e^p |\nabla u|^p + C\int_{\e(z+Y'_*) \cap \Omega} |u|^p.
    \end{equation*}
    Summing over $z$, we get
    \begin{equation}\label{est.MultiscalePoincare0}
        \int_{\Omega_\e} |u|^p \le C\e^p \int_{\Omega_\e} \phi_\e^p |\nabla u|^p + C\int_{\Omega'_\e} |u|^p \le C\int_{\Omega_\e} \phi_\e^p|\nabla u|^p,
    \end{equation}
    where we have used \eqref{est.up2phiDup} in the last inequality. The proof is complete.
\end{proof}

\begin{remark}
    The inequality (taken from \eqref{est.MultiscalePoincare0})
    \begin{equation}\label{est.Multiscale.Poincare}
        \int_{\Omega_\e} |u|^p \le C\e^p \int_{\Omega_\e} \phi_\e^p |\nabla u|^p + C\int_{\Omega'_\e} |u|^p
    \end{equation}
    can be viewed as a version of the so-called multiscale Poincar\'{e} inequality. The $\e$ appearing before the gradient term is not classical and it is related the multiscale feature of the inequality.
\end{remark}

\begin{theorem}
[Embedding from weighted to unweighted II]\label{thm.unweigthed2}
    Let $1< p<\infty$. There exists $C>0$ such that for any $u\in W^{1,p}_{\phi_\e}(\Omega_\e)$,
    \begin{equation}\label{est.uLp}
        \| u \|_{L^p(\Omega_\e)} \le C\e \| \phi_\e \nabla u \|_{L^p(\Omega_\e)} + C \| \phi_\e u \|_{L^p(\Omega_\e)},
    \end{equation}
    and
    \begin{equation}\label{est.u-L.Lp}
        \inf_{L\in \R} \| u - L \|_{L^p(\Omega_\e)} \le C\| \phi_\e \nabla u \|_{L^p(\Omega_\e)}.
    \end{equation}
\end{theorem}

\begin{proof}
    The inequality \eqref{est.uLp} follows from \eqref{est.Multiscale.Poincare} and the fact that $\phi_\e \approx 1$ in $\Omega'_\e$. We  only need to prove \eqref{est.u-L.Lp}. Note that if $\tilde{u} = \mathcal{E}_\e(u)$ is the harmonic extension of $u$, then $\tilde{u} - L$ is the harmonic extension of $u-L$. We set
    \begin{equation}\label{eq.L=aveu}
        L = \fint_{\Omega} \tilde{u}. 
    \end{equation}
    We mimic the proof of Theorem \ref{thm.unweigthed}. In fact,
    \begin{equation*}
        \begin{aligned}
            \int_{\Omega_\e} |u-L|^p & \le C\e^p \int_{\Omega_\e} \phi_\e^p |\nabla u|^p + C\int_{\Omega'_\e} |u-L|^p \\
            & \le C\e^p \int_{\Omega_\e} \phi_\e^p |\nabla u|^p + C\int_{\Omega} |\tilde{u}-L|^p \\
            & \le C\e^p \int_{\Omega_\e} \phi_\e^p |\nabla u|^p + C\int_{\Omega} |\nabla \tilde{u}|^p \\
            & \le C\e^p \int_{\Omega_\e} \phi_\e^p |\nabla u|^p + C\int_{\Omega_\e} \phi_\e^p |\nabla u|^p \\
            & \le C\int_{\Omega_\e} \phi_\e^p |\nabla u|^p,
        \end{aligned}
    \end{equation*}
    where we have used the classical Poincar\'{e} inequality in the third inequality.
\end{proof}

\begin{theorem}[Weighted Sobolev-Poincar\'{e} inequality I] \label{thm.WSPI}
Let $1< p<\infty$ and $u\in W^{1,p}_{\phi_\e,0}(\Omega_\e)$.
\begin{itemize} 
    \item[(i)]
    If $1< p<d$ and $\frac{1}{p^*} =\frac{1}{p}-\frac{1}{d}$, there exists $C>0$ such that
        \begin{equation}\label{est.uq.Dup}
            \| \phi_\e u \|_{L^{p^*}(\Omega_\e)} \le C\| \phi_\e \nabla u \|_{L^p(\Omega_\e)}.
        \end{equation}
    \item[(ii)] If $p>d$, there exists $C>0$ such that
    \begin{equation*}
        \| \phi_\e u \|_{L^\infty(\Omega_\e)} \le C\| \phi_\e \nabla u \|_{L^p(\Omega_\e)}.
    \end{equation*}
\end{itemize}
\end{theorem}
\begin{proof}
    (i) We first prove
    \begin{equation}\label{est.u.Du}
        \| u \|_{L^{p^*}(\Omega'_\e)} \le C\| \phi_\e \nabla u \|_{L^p(\Omega_\e)}.
    \end{equation}
    Let $\tilde{u} = \mathcal{E}_\e(u)$ be the harmonic extension of $u$ from $\Omega'_\e$ to the entire $\Omega$ constructed as before. Then we have \eqref{est.extu}. Since $u = \tilde{u}$ in $\Omega'_\e$,
    by the classical Sobolev-Poincar\'{e} ienquality in $\Omega$ and \eqref{est.extu}, we obtain
    \begin{equation}\label{est.uLq.DuLp}
        \| u\|_{L^{p^*}(\Omega'_\e)} = \| \tilde{u}\|_{L^{p^*}(\Omega'_\e)} \le C\| \nabla \tilde{u} \|_{L^p(\Omega)} \le C\| \phi_\e \nabla u \|_{L^p(\Omega_\e)}.
    \end{equation}

    Next, we show
    \begin{equation}\label{est.u.Du+u}
        \| \phi_\e u \|_{L^{p^*}(\Omega_\e)} \le C\| \phi_\e \nabla u \|_{L^p(\Omega_\e)} + C\| u\|_{L^{p^*}(\Omega'_\e)}.
    \end{equation}
    We will again use the technique of ``multiscale inequalities''. Precisely,
    \begin{equation}
    \begin{aligned}\label{est.phieu.Lq}
        \int_{\Omega_\e} |\phi_\e u|^{p^*} & = \sum_{z} \int_{\e(z+Y_*)\cap \Omega_\e} |\phi_\e u|^{p^*} \\
        &= \sum_z \e^d \int_{(z+Y_*)\cap \frac{1}{\e}\Omega} |\phi(x) u(\e x)|^{p^*} dx \\
        & \le \sum_z C \e^d \bigg( \int_{(z+Y_*)\cap \frac{1}{\e}\Omega} |\phi(x) \e \nabla u(\e x)|^p dx \bigg)^{{p^*}/p} \\
        &\qquad\qquad + \sum_z C\e^d \bigg( \int_{(z+Y'_*)\cap \frac{1}{\e}\Omega} | u(\e x)|^p dx \bigg)^{{p^*}/p} \\
        & \le \sum_{z} C\bigg( \int_{\e(z+Y_*)\cap \Omega_\e} |\phi_\e \nabla u|^p \bigg)^{{p^*}/p} + \sum_z C\int_{\e(z+Y'_*)\cap \Omega_\e} |u|^{p^*},
    \end{aligned}
    \end{equation}
    where we have used \eqref{est.local Sobolev} in the third line. Now we recall the following elementary
    inequality: for any nonnegative sequence $a_i \ge 0$ and $\beta \ge 1$,
    \begin{equation}\label{est.basic ineq}
        \sum_i a_i^\beta  \le \Big( \sum_i \alpha_i \Big)^\beta.
    \end{equation}
    Applying this inequality to the first term of the last line of \eqref{est.phieu.Lq}, we obtain
    \begin{equation*}
        \int_{\Omega_\e} |\phi_\e u|^{p^*} \le C\bigg( \int_{\Omega_\e} |\phi_\e \nabla u|^p \bigg)^{{p^*}/p} + C\int_{\Omega'_\e} |u|^{p^*}.
    \end{equation*}
    This proves \eqref{est.u.Du+u}. Clearly, the estimates \eqref{est.u.Du} and \eqref{est.u.Du+u} together yield \eqref{est.uq.Dup}.

    (ii) If $p>d$,  \eqref{est.uLq.DuLp} is replaced by
    \begin{equation*}
        \| u\|_{C^\alpha(\Omega'_\e)} \le C\| \phi_\e \nabla u \|_{L^p(\Omega_\e)},
    \end{equation*}
    where $\alpha = 1-\frac{d}{p}$.

    Now consider in each periodic cell $\e(z+Y_*)$ (the estimate for the boundary cells is the same), and apply \eqref{est.Ca.phiu} in $z+Y_*$ after scaling,
    \begin{equation*}
    \begin{aligned}
        \sup_{\e(z+Y_*)} |\phi_\e u| & = \sup_{x\in z+Y_*} |\phi(x) u(\e x)| \\
        & \le C\| \phi \nabla(u(\e x)) \|_{L^p(z+Y_*)} + C\| u(\e x) \|_{L^p(Y'_*)} \\
        & \le C\e^{1-d/p} \| \phi_\e \nabla u \|_{L^p(\e(z+Y_*))} + C\| u \|_{L^\infty(\e(z+Y'_*))} \\
        & \le C\e^{1-d/p} \| \phi_\e \nabla u \|_{L^p(\Omega_\e)} + C\|u \|_{C^\alpha(\Omega'_\e)} \\
        & \le C\| \phi_\e \nabla u \|_{L^p(\Omega_\e)}.
    \end{aligned} 
    \end{equation*}
    This ends the proof.
\end{proof}

\begin{theorem}[Weighted Sobolev-Poincar\'{e} inequality II]
    Let $1<  p<\infty$ and $u\in W^{1,p}_{\phi_\e}(\Omega_\e)$.
\begin{itemize} 
    \item[(i)]
    If $1< p<d$, there exists $C>0$ such that
        \begin{equation}\label{est.u-Lq.Dup}
            \inf_{L \in \R} \| \phi_\e (u - L) \|_{L^{p^*}(\Omega_\e)} \le C\| \phi_\e \nabla u \|_{L^p(\Omega_\e)},
        \end{equation}
    and
    \begin{equation}\label{est.pheu.Sobolev}
        \| \phi_\e u \|_{L^{p^*}(\Omega_\e)} \le C\| \phi_\e \nabla u \|_{L^p(\Omega_\e)} + C\| \phi_\e  u \|_{L^p(\Omega_\e)}.
    \end{equation}
    \item[(ii)] If $p>d$, there exists $C>0$ such that
    \begin{equation*}
        \inf_{L\in \R} \| \phi_\e (u-L) \|_{L^\infty(\Omega_\e)} \le C\| \phi_\e \nabla u \|_{L^p(\Omega_\e)},
    \end{equation*}
    and
    \begin{equation*}
        \| \phi_\e u \|_{L^\infty(\Omega_\e)} \le C\| \phi_\e \nabla u \|_{L^p(\Omega_\e)} +  C\| \phi_\e u \|_{L^p(\Omega_\e)}.
    \end{equation*}
\end{itemize}
\end{theorem}
\begin{proof}
    For \eqref{est.u-Lq.Dup}, the proof is almost identical to Theorem \ref{thm.WSPI} with replacing $u$ by $u-L$. The only difference is that when we use the classical Sobolev-Poincar\'{e} inequality in \eqref{est.uLq.DuLp}, we take $L$ to be the average of $\tilde{u}$ over $\Omega$ (as in \eqref{eq.L=aveu}) to 
    obtain 
    \begin{equation*}
        \| u-L \|_{L^{p^*}(\Omega'_\e)} \le \| \tilde{u} - L \|_{L^{p^*}(\Omega_\e)} \le C\| \nabla \tilde{u} \|_{L^p(\Omega)} \le C\| \phi_\e \nabla u \|_{L^p(\Omega_\e)}.
    \end{equation*}
    The rest of the proof is exactly the same.

    For \eqref{est.pheu.Sobolev}, we replace \eqref{est.uLq.DuLp} by
    \begin{equation*}
    \begin{aligned}
        \| u\|_{L^{p^*}(\Omega'_\e)}   = \| \tilde{u}\|_{L^{p^*}(\Omega'_\e)} 
        & \le C\| \nabla \tilde{u} \|_{L^p(\Omega)} + \| \tilde{u} \|_{L^p(\Omega'_\e)}\\
        &\le C\| \phi_\e \nabla u \|_{L^p(\Omega_\e)} + C\| \phi_\e u \|_{L^p(\Omega_\e)}.
    \end{aligned}
    \end{equation*}
    The rest of the proof is exactly the same. The proof for (ii) is similar.
\end{proof}

\begin{remark}[Compact embedding]
    Let $1< p<d$ and $\frac{1}{p^*} = \frac{1}{p} - \frac{1}{d}$. Then for any $1\le q<p^*$, the embedding $W^{1,p}_{\phi_\e}(\Omega_\e) \to L^q_{\phi_\e}(\Omega_\e)$ is compact. This general fact 
    is not used in this paper.
\end{remark}

\section{Regularity of degenerate equations at small scales}\label{sec.3}

Let $V_\e= W^{1, 2}_{\phi_\e, 0} (\Omega_\e)$.
For $f\in L^2(\Omega_\e)^d$ and $F\in L^2(\Omega_\e)$, we call $u_\e\in V_\e$ a weak solution to the boundary value 
problem, 
\begin{equation}\label{BVP1}
\cL_\e( u_\e)=\text{\rm div}(\phi_\e f) +F
\quad \text{ in } \Omega_\e \quad \text{ and } \quad u_\e =0 \quad 
\text{ on } \partial \Omega,
\end{equation}
if for any $v \in V_\e$,
\begin{equation*}
    \int_{\Omega_\e} \phi_\e^2 A_ \e \nabla u_\e \cdot \nabla v
    =-\int_{\Omega_\e} \phi_\e f \cdot \nabla v
    + \int_{\Omega_\e} F  v.
\end{equation*}
The existence and uniqueness of the weak solution $u_\e$ follow readily from the Lax-Milgram Theorem by using the inequality 
$\| v \|_{L^2(\Omega_\e)} \le C \|\phi_\e \nabla  v \|_{L^2(\Omega_\e)}$ for $v\in V_\e$; see Theorem \ref{thm.WSPI}.
Moreover, the solution satisfies the energy estimate,
\begin{equation*}
    \| \phi_\e \nabla u_\e \|_{L^2(\Omega_\e)} 
    \le C \left\{ 
    \| f \|_{L^2(\Omega_\e)} + \| F \|_{L^2(\Omega_\e)} \right\}.
\end{equation*}

In this section, we focus on the local regularity at small scales for the special equation $\cL_\e(u_\e) = \phi_\e F$ in $\Omega_\e$. By setting $u(x) = u_\e(\e x)$, this equation can be rescaled to
\begin{equation}\label{eq.rescaled.phiF}
    -\text{div}(\phi^2 A \nabla u) = \e^2 \phi(x) F(\e x) \quad \text{in } \e^{-1}\Omega_\e.
\end{equation}
In order to establish
the small-scale estimates of the equation \eqref{eq.rescaled.phiF} for all the points in $\e^{-1} \Omega_\e$, we need to consider three cases separately. 

\begin{itemize}
    \item \textbf{Interior case}: $x_0 \in \e^{-1}\Omega_\e$ and $B_{4r}(x_0) \subset \e^{-1}\Omega_\e$. In this case, for any $x\in B_{2r}(x_0)$, $\phi(x) \approx \phi(x_0)$. Thus, the equation \eqref{eq.rescaled.u} in $B_{2r}(x_0)$ can be reduced to a uniformly elliptic equation and the classical interior estimate applies.

    \item \textbf{Boundary case I}: $x_0 \in \e^{-1} \partial \Omega$ and $B_{4r}(x_0) \cap \e^{-1}\Omega \subset \e^{-1} \Omega_{\e}$. In this case $B_{4r}(x_0)$ does not intersect with the holes in $\Omega_\e$. Hence, for $x\in B_{2r}(x_0) \cap \e^{-1} \Omega$, $\phi(x) \approx \phi(x_0)$. Then the classical boundary estimate of elliptic equations applies.

    \item \textbf{Boundary case II}: $x_0 \in \e^{-1}\partial \Omega_\e \setminus \e^{-1}\partial \Omega$, i.e., $x_0$ is on the boundary of holes. For sufficiently small and
    fixed $r_0<1$ and $r<r_0$, we have $B_{2r}(x_0)$ only intersects with the boundary of one hole and does not get close to the other holes. 
    Note that in this case, the coefficient $\phi(x)$ degenerates as a distance function from the hole and the classical elliptic theory does not apply.
\end{itemize}

As the interior case and the boundary case I are classical, we only need to focus on the regularity estimates of the boundary case II. The corresponding estimates in each cell $Y^*$ (including the boundary cells) may be derived by a covering argument.
    
Let $|z|_\infty := \max_{1\le j\le d} |z_j|$ for $z = (z_1,z_2,\dots, z_d) \in \Z^d$. Let $Y_*^+ = \cup_{|z|_\infty\le 1} (z+Y_*)$ be the enlarged 
cell of $Y_*$. In view of \eqref{eq.rescaled.phiF}, we
consider the rescaled equation in an enlarged cell
\begin{equation}\label{eq.rescaled.u}
    -\text{\rm div}  (\phi^2 A \nabla u) = \phi F \quad \text{in } Y_*^+.
\end{equation}
Let $u\in W^{1, 1}_{\rm loc}(Y_*^+)$. We call $u$ a weak solution of \eqref{eq.rescaled.u} if 
$\phi u \in L^2(Y_*^+)$, $\phi |\nabla u | \in L^2(Y_*^+)$, and for any $v\in C_0^\infty (Y^+)$, we have
$$
\int_{Y_*^+} \phi^2 A \nabla u \cdot \nabla v =\int_{Y_*^+} \phi F v,
$$
where $Y^+=\cup_{|z|_\infty\le 1} (z+Y)$.

{We first prove the local boundedness of weak solutions  by the Moser's iteration.}

\begin{lemma}\label{lem.Moser}
    Let $u$ be a weak solution of \eqref{eq.rescaled.u} and $F\in L^p(Y_*)$ for some $p>d$. 
        If $B$ is centered on $\partial T$ and $r=\text{\rm diam}(B) \le c_0/4$, then
        \begin{equation}\label{local-moser}
            \sup_{\frac12 B \cap Y_*} |u| \le Cr^{-1} \bigg( \fint_{B\cap Y_*} |\phi u|^2 \bigg)^{1/2} + Cr \bigg( \fint_{B\cap Y_*} |F|^p \bigg)^{1/p}.
        \end{equation}
\end{lemma}

\begin{proof}
    By rescaling, it suffices to consider the case
    $\text{diam}(B) = 1$. Let $B_r = rB$ and $\widetilde{B}_r = B_r \cap Y_*$. Let $1/2<r<R<1$ and $\eta \in C_0^\infty (B_R)$ be such that $0\le \eta \le 1$ in $B_R$, $\eta = 1$ in $B_r$ and $|\nabla \eta| \le C(R-r)^{-1}$. Let $m\ge 2$, $\ell\ge 1$,  and 
    $$
    u_\ell^+ = \min\{|u|, \ell\}, \qquad v_\ell =(u_\ell^+)^{m-2} u.
    $$
    Note that 
    $$
    \nabla v_\ell = (m-1) (u_\ell^+)^{m-2}   \nabla u \chi_{\{ |u|< \ell\} }
    + (u_\ell^+)^{m-2}  \nabla u \chi_{\{ |u| \ge \ell\} }.
    $$
    By a limiting argument, it is not hard to show that we may 
    test  the equation \eqref{eq.rescaled.u} against $v_\ell  \eta^2$.
    It follows that 
    \begin{equation}\label{est.um}
    \aligned
         & (m-1)\int_{\widetilde{B}_R} \phi^2( A\nabla u \cdot \nabla u) \chi_{\{ |u|<\ell\} }   (u_\ell^+)^{m-2} \eta^2 
         + \int_{\widetilde{B}_R} \phi^2 (A\nabla u \cdot \nabla u) \chi_{\{ |u|\ge \ell\} } (u_\ell^+)^{m-2} \eta^2\\
         &= \int_{\widetilde{B}_R} \phi F v_\ell  \eta^2 
        - 2\int_{\widetilde{B}_R} \phi^2 (A \nabla u\cdot \nabla \eta) v_\ell  \eta.
        \endaligned
    \end{equation}
    Let 
    $$
    w_\ell =(u_\ell^+)^{\frac{m}{2} - 1}  u\eta .
    $$
    Then
    \begin{equation} \label{eq.Dwl}
    \nabla w_\ell=\frac12 m  (u_\ell^+)^{\frac{m}{2}-1}  (\nabla u)\eta  \chi_{\{ |u|< \ell\} }
    + (u_\ell^+)^{\frac{m}{2}-1} (\nabla u)\eta \chi_{\{ |u| \ge \ell\} } 
+ (u_\ell^+)^{\frac{m}{2} - 1} u\nabla \eta.
    \end{equation}
    This, together with \eqref{est.um}, gives
    \begin{equation}\label{est.um-1}
    \aligned
         \int_{\widetilde{B}_R} \phi^2 |\nabla w_\ell  |^2 
         & \le C m  \int_{\widetilde{B}_R}  \phi | F| |v_\ell| \eta^2 
         + Cm   \int_{\widetilde{B}_R} \phi^2 |\nabla u| |\nabla \eta| |v_\ell| \eta\\
         &\qquad + C\int_{\widetilde{B}_R} \phi^2 (u_\ell^+)^{m-2} u^2 |\nabla \eta|^2.
        \endaligned
  \end{equation}
    Using H\"{o}lder's inequality and Young's inequality, we deduce that 
    \begin{equation}\label{est.Fum-1}
    \begin{aligned}
        &  m \int_{\widetilde{B}_R} \phi |F| |v_\ell| \eta^2  
        \le \| \delta F \|_{L^p(\widetilde{B}_R)} \| (\delta m)^{-1} |u_\ell^+|^{\frac{m-2}{2}}  \eta \|_{L^2(\widetilde{B}_R)} \| m^2 \phi w_\ell  \|_{L^{\gamma}(\widetilde{B}_R)} \\
        & \le C \delta^m \| F \|_{L^p(\widetilde{B}_R)}^{m} + \big( \frac{1}{m\delta} \big)^{2m/(m-2)} \| |u_\ell^+|^{\frac{m-2}{2}} \eta \|_{L^2(\widetilde{B}_R)}^{2m/(m-2)} + m^4 \| \phi w_\ell  \|_{L^{\gamma}(\widetilde{B}_R)}^2,
    \end{aligned}
    \end{equation}
    where $\frac{1}{p} + \frac12 + \frac{1}{\gamma} = 1$ and $\delta = \delta(m) > 0$ is a suitable number to be determined. Since $p>d$, we have  $2<\gamma<2^*$ (an obvious modification is needed for $d=2$ through the proof). 
    By  H\"{o}lder's inequality and the boundedness of $\eta$,
    \begin{equation}\label{est.Fum-2}
        \| |u_\ell^+|^{\frac{m-2}{2}} \eta \|_{L^2(\widetilde{B}_R)}^{2m/(m-2)} \le \| |u_\ell^+|^{\frac{m}{2}} \eta \|_{L^2(\widetilde{B}_R)}^{2} \le \| w_\ell \|_{L^2(\widetilde{B}_R)}^{2}.
    \end{equation}
    By an interpolation inequality,
    \begin{equation}\label{est.Fum-3}
        \| \phi w_\ell \|_{L^{\gamma}(\widetilde{B}_R)}^2 \le \big( \frac{1}{m^5 \delta} \big) \| \phi w_\ell \|_{L^{2^*}(\widetilde{B}_R)}^2 + C (m^5 \delta)^\alpha \| \phi w_\ell \|_{L^{2}(\widetilde{B}_R)}^2,
    \end{equation}
    where $\alpha>0$ is a constant depending only on $p$ and $d$.
    Substituting \eqref{est.Fum-2} and \eqref{est.Fum-3} into \eqref{est.Fum-1}, we have
    \begin{equation}\label{est.Fum-4}
    \aligned
         m\int_{\widetilde{B}_R} \phi |F| |v_\ell| \eta^2 
        &\le C \delta^m \| F \|_{L^p(\widetilde{B}_R)}^{m} + C m^4 (m^5 \delta)^\alpha \| \phi w_\ell \|_{L^{2}(\widetilde{B}_R)}^2 \\
        &\qquad + \big( \frac{1}{m\delta} \big)^{2m/(m-2)}  \| w_\ell \|_{L^2(\widetilde{B}_R)}^{2} + \frac{1}{m\delta} \| \phi w_\ell \|_{L^{2^*}(\widetilde{B}_R)}^2 .
        \endaligned
    \end{equation}

    On the other hand, using $\eqref{eq.Dwl}$, we have
    $$
    (u_\ell^+)^{\frac{m}{2}-1} |\nabla u | \eta 
    \le |\nabla w_\ell - (u_\ell^+)^{\frac{m}{2} -1} u \nabla \eta  |.
    $$
    This allows us to estimate the second integral on the right-hand side of \eqref{est.um-1} as follows,
    \begin{equation}\label{est.um-3}
        m   \int_{\widetilde{B}_R} \phi^2 |\nabla u| |\nabla \eta| |v_\ell| \eta \le  \frac{1}{m\delta} \| \phi \nabla w_\ell \|^2_{L^{2}(\widetilde{B}_R)} + C(m^3\delta+m)  \int_{\widetilde{B}_R} \phi^2 (u_\ell^+)^{m-2} u^2 |\nabla \eta|^2.
    \end{equation}

    Now, inserting \eqref{est.Fum-4} and \eqref{est.um-3} into \eqref{est.um-1}, we obtain
    \begin{equation}\label{est.um-4}
    \begin{aligned}
        & \int_{\widetilde{B}_R} \phi^2 |\nabla w_\ell |^2 \\ 
        & \le C \delta^m \| F \|_{L^p(\widetilde{B}_R)}^{m}+ C(m^3\delta+ m^4(m^5 \delta)^\alpha+m)  \int_{\widetilde{B}_R} \phi^2 (u_\ell^+)^{m-2} u^2 (|\nabla \eta|^2 + \eta^2) \\
        & \qquad + C\big( \frac{1}{m\delta} \big)^{2m/(m-2)}  \| w_\ell \|_{L^2(\widetilde{B}_R)}^{2} + \frac{C}{m\delta} \| \phi w_\ell \|_{L^{2^*}(\widetilde{B}_R)}^2 + \frac{C}{m\delta} \| \phi \nabla w_\ell \|^2_{L^{2}(\widetilde{B}_R)}
    \end{aligned}
    \end{equation}
    Note $w_\ell = 0$ on $\partial B_R \cap Y_*$. Recall two inequalities given by Lemma \ref{lem.localSobolev} (i) and Lemma \ref{lem.local Poincare},
    \begin{equation}
        \| \phi w_\ell \|_{L^{2^*}(\widetilde{B}_R)} \le C \| \phi \nabla w_\ell \|_{L^2(\widetilde{B}_R)},
    \end{equation}
    and
    \begin{equation}
        \| w_\ell \|_{L^{2}(\widetilde{B}_R)} \le C \| \phi \nabla w_\ell \|_{L^2(\widetilde{B}_R)}.
    \end{equation}
    Choosing $\delta = C_1 /m$ with sufficiently large $C_1 >0$, we obtain from \eqref{est.um-4} that
    \begin{equation}\label{est.um-5}
        \int_{\widetilde{B}_R} \phi^2 |\nabla w_\ell |^2 \le \frac{C^m}{m^m} \| F \|_{L^p(\widetilde{B}_R)}^{m} + Cm^{4+4\alpha} \int_{\widetilde{B}_R} \phi^2 (u_\ell^+)^{m-2} u^2 (|\nabla \eta|^2 + \eta^2).
    \end{equation}
    By  H\"{o}lder's inequality, we have
    \begin{equation}\label{est.phiDw.interp}
        \bigg( \int_{\widetilde{B}_R} \phi^2 |w_\ell|^{2\gamma_0} \bigg)^{1/2\gamma_0} \le \| \phi w_\ell \|_{L^{2^*}(\widetilde{B}_R)}^\theta \| w_\ell \|_{L^{2}(\widetilde{B}_R)}^{1-\theta} \le C \| \phi \nabla w_\ell \|_{L^2(\widetilde{B}_R)}, 
    \end{equation}
    where $\theta=d/(d+2)$ and $\gamma_0 = 1+2/d>1$.
    Hence, it follows from \eqref{est.um-5} and \eqref{est.phiDw.interp} that 
    \begin{equation}\label{est.MoserIter}
    \begin{aligned}
        & \bigg( \int_{\widetilde{B}_r} \phi^2 (u_\ell^+)^{(m-2)\gamma_0} |u|^{2\gamma_0} \bigg)^{1/\gamma_0 m} \\
        & \le \frac{C}{m} \| F \|_{L^p(\widetilde{B}_R)} + \bigg\{ \frac{ C m^{4+4\alpha} }{(R-r)^2} \bigg\}^{1/m}  \bigg( \int_{\widetilde{B}_R} \phi^2 (u_\ell^+)^{m-2} u^2 \bigg)^{1/m}.
    \end{aligned}
    \end{equation}
    The above estimate holds for any $\ell>1, m\ge 2$ and $1/2<r<R<1$. Moreover, the constant $C$ is independent of $\ell, m, R$ and $r$. By an iteration starting from $m=2$ and letting $\ell \to \infty$, we can see that $\int_{\widetilde{B}_r} \phi^2 |u|^{m\gamma_0}$ is finite for any $r<1$ and $m\ge 2$. Hence, sending $\ell \to \infty$ in \eqref{est.MoserIter}, we get
    \begin{equation}\label{est.MoserIter-2}
    \begin{aligned}
        \bigg( \int_{\widetilde{B}_r} \phi^2 |u|^{\gamma_0 m} \bigg)^{1/\gamma_0 m}  \le \frac{C}{m} \| F \|_{L^p(\widetilde{B}_R)} + \bigg\{ \frac{ Cm^{4+4\alpha}  }{(R-r)^2} \bigg\}^{1/m}  \bigg( \int_{\widetilde{B}_R} \phi^2 |u|^{m} \bigg)^{1/m}.
    \end{aligned}
    \end{equation}

    Now, we  iterate the inequality \eqref{est.MoserIter-2} by choosing an infinite sequence of shrinking radii $r_k$ with $1/2< r_0 \le 1$ and $\lim_{k\to \infty} r_k = 1/2$. Let $\theta \in (0,1)$ (particularly, we may pick $\theta = 1/2$ in the following calculation). Define
    \begin{equation*}
        r_k = \frac12 + \frac12 (1-\theta) \theta^k.
    \end{equation*}
    and
    \begin{equation*}
        p_0 = 2, \quad p_{k+1} = \gamma_0 p_k=2 \gamma_0^{k+1}.
    \end{equation*}
    Applying \eqref{est.MoserIter-2} with $m = p_k$ and $R = r_{k}, r= r_{k+1}$, we obtain
    \begin{equation*}
        \bigg( \int_{\widetilde{B}_{r_{k+1}}} \phi^2 |u|^{p_{k+1}} \bigg)^{1/p_{k+1}} \le \frac{C}{p_k} \| F \|_{L^p(\widetilde{B}_{r_k})} + \bigg\{ \frac{Cp_k^{2+ 2\alpha} )}{(1-\theta)^2 \theta^{k} } \bigg\}^{2/p_k} \bigg( \int_{\widetilde{B}_{r_k}} \phi^2 |u|^{p_k} \bigg)^{1/p_k} .
    \end{equation*}
    By iteration, we get
    \begin{equation*}
        \bigg( \int_{\widetilde{B}_{r_{k+1}}} \phi^2 |u|^{p_{k+1}} \bigg)^{1/p_{k+1}} \le \bigg\{ \sum_{j=0}^k  \frac{C D_k }{2\gamma_0^{k-j}}  \bigg\} \| F \|_{L^p(\widetilde{B}_{r_0})}  + D_k \bigg( \int_{\widetilde{B}_{r_0}} \phi^2 |u|^{2} \bigg)^{1/2},
    \end{equation*}
    where 
    \begin{equation*}
        D_k = \prod_{j=0}^k \bigg\{ \frac{C \gamma_0^{(2+2\alpha)j } }{(1-\theta)^2 \theta^j } \bigg\}^{1/\gamma_0^j}.
    \end{equation*}
    Due to the fact $\gamma_0>1$, we have $\sum_j (j /\gamma_0^j )< \infty$.
     It follows that $D_k \le C_\theta$ independent of $k$. Hence, 
    \begin{equation*}
        \sum_{j=0}^k  \frac{C D_k }{2\gamma_0^{k-j}} \le C_\theta.
    \end{equation*}
    Consequently,
    \begin{equation*}
        \bigg( \int_{\widetilde{B}_{r_{k+1}}} \phi^2 |u|^{p_{k+1}} \bigg)^{1/p_{k+1}} \le C_\theta \big( \| F \|_{L^p(\widetilde{B}_{r_0})} + \| \phi u \|_{L^2(\widetilde{B}_{r_0})} \big).
    \end{equation*}
    Finally, taking $k\to \infty$ and noting that $r_k> 1/2$, we derive
    \begin{equation*}
        \| u \|_{L^\infty(\widetilde{B}_{1/2})} \le C_\theta \big( \| F \|_{L^p(\widetilde{B}_{r_0})} + \| \phi u \|_{L^2(\widetilde{B}_{r_0})} \big).
    \end{equation*}
    This proves the desired estimate.
\end{proof}

\begin{lemma}\label{lem.SS.Lip}
    Let $u$ be a weak solution of \eqref{eq.rescaled.u} and $F\in L^p(Y_*)$ for some $p>d$. 
        If $B$ is centered on $\partial T$ and $\text{\rm diam}(B) \le c_0/4$, then
        \begin{equation}
            \sup_{\frac12 B \cap Y_*} |\phi \nabla u| \le C\bigg( \fint_{B\cap Y_*} |\phi \nabla u|^2 \bigg)^{1/2} + Cr \bigg( \fint_{B\cap Y_*} |F|^p \bigg)^{1/p},
        \end{equation}
        where $C$ depends on $p, \phi, A$ and $T$. 
\end{lemma}
\begin{proof}
Let $r = \text{diam}(B)$. 
Let ${v} = \phi u$. 
Using the assumption that $A$ is symmetric, a computation shows that
\begin{equation*}
    \text{\rm div} (A\nabla v)
    =\text{\rm div} (A\nabla \phi) u
    + \phi^{-1} \text{div} (\phi^2 A \nabla u).
    \end{equation*}
As a result,  ${v}$ satisfies
\begin{equation}\label{eq.tv.Laplace}
    -\text{\rm div} (A\nabla v) = V u + F \quad \text{ in } Y_*,
\end{equation}
where 
$
V =- \text{\rm div} (A\nabla \phi).
$
Note that $v = 0$ on $\partial T$ in the sense of trace. 
Under the assumptions that $A$ is H\"older continuous and $V\in L^{q}(Y_*)$ for some $q>d$, it follows from \eqref{eq.tv.Laplace} and
 the standard elliptic regularity theory
\begin{equation}\label{est.tv.Linfinity}
\begin{aligned}
    \sup_{ \frac12 B \cap Y_*}|\nabla {v}| & \le Cr \bigg( \fint_{B\cap Y_*} |V u|^q \bigg)^{1/q} + Cr \bigg( \fint_{B\cap Y_*} |F|^p \bigg)^{1/p} \\
    & \le  C r \bigg( \fint_{B\cap Y_*} |V|^q \bigg)^{1/q} \| u\|_{L^\infty(B\cap Y_*)} + Cr \bigg( \fint_{B\cap Y_*} |F|^p \bigg)^{1/p} \\
    & \le C \bigg( \fint_{B\cap Y_*} |V|^q \bigg)^{1/q} \bigg( \fint_{2B\cap Y_*} |\phi u|^2 \bigg)^{1/2}  \\
    & \qquad + Cr \bigg\{ 1 + r\bigg( \fint_{B\cap Y_*} |V|^q \bigg)^{1/q} \bigg\} \bigg( \fint_{2B\cap Y_*} |F|^p \bigg)^{1/p},
\end{aligned}
\end{equation}
where $p,q>d$ and we have used Lemma \ref{lem.Moser} in the last inequality.


Note that $\nabla v= (\nabla \phi) u + \phi \nabla  u $. Thus, by \eqref{est.tv.Linfinity} and \eqref{local-moser},
\begin{equation}\label{est.phiDv.Y*}
\begin{aligned}
    \sup_{\frac12 B \cap Y_*} |\phi \nabla u| & \le \sup_{\frac12 B \cap Y_*} | (\nabla \phi ) u| + \sup_{\frac12 B \cap Y_*} |\nabla v| \\
    &  \le C \bigg\{ r^{-1}+ \bigg( \fint_{B\cap Y_*} |V|^q \bigg)^{1/q} \bigg\} \bigg( \fint_{2B\cap Y_*} |\phi u|^2 \bigg)^{1/2}  \\
    & \qquad + C r \bigg\{ 1 + r\bigg( \fint_{B\cap Y_*} |V|^q \bigg)^{1/q} \bigg\} \bigg( \fint_{2B\cap Y_*} |F|^p \bigg)^{1/p}.
\end{aligned}
\end{equation}
Finally, observe that  $u-L$ satisfies the same equation
\eqref{eq.rescaled.u} for any constant $L$. Thus $\phi u$ in \eqref{est.phiDv.Y*} can be replaced by $\phi(u-L)$ for any $L$. Moreover, it follows by a similar argument as in the proof of
the Poincar\'{e} inequality \eqref{est.u-L.Lp} that 
\begin{equation}
    \inf_{L\in \R}\bigg( \fint_{2B\cap Y_*} |\phi (u-L) |^2 \bigg)^{1/2} \le Cr\bigg( \fint_{3B\cap Y_*} |\phi \nabla u|^2 \bigg)^{1/2}.
\end{equation}
It follows from \eqref{est.phiDv.Y*} that
\begin{equation}\label{est.phiDv.Linfinity}
\begin{aligned}
    \sup_{\frac12 B \cap Y_*} |\phi \nabla u| &  \le C \big\{ 1+ \| V\|_{L^q(Y_*)} \big\} \bigg( \fint_{3B\cap Y_*} |\phi \nabla u|^2 \bigg)^{1/2}  \\
    & \qquad + C r \big\{ 1 + \| V\|_{L^q(Y_*)} \big\} \bigg( \fint_{3B\cap Y_*} |F|^p \bigg)^{1/p}. 
\end{aligned}
\end{equation}
By a simple covering argument, on may replace  $3B\cap Y_*$ in the right-hand side of \eqref{est.phiDv.Linfinity} by $B\cap Y_*$.
\end{proof}


\begin{theorem}\label{lem.Local Linfinity}
Let $u_\e\in V_\e$ be a solution of $
\mathcal{L}_\e (u_\e) =\phi_\e F  \text{ in } \e (z+Y_*^+)
$
for some $z\in \mathbb{Z}^d$.
Then for $p>d$,
\begin{equation}\label{est.local.Qe}
    \sup_{ \e(z+Y_*)} |u_\e| \le C
    \bigg( \fint_{\e(z+Y^+_*)} |\phi_\e u_\e|^2 \bigg)^{1/2} +C\e^2 \bigg(
    \fint_{ \e(z+Y^+_*)} |F|^p \bigg)^{1/p}, 
\end{equation}
and
\begin{equation}\label{est.local.Lip}
    \sup_{\e(z+Y_*)} |\phi_\e \nabla u_\e| \le C\bigg( \fint_{ \e(z+Y^+_*)} |\phi_\e \nabla u_\e|^2 \bigg)^{1/2} +C\e  \bigg( \fint_{ \e(z+Y^+_*)} |F|^p \bigg)^{1/p}.
\end{equation}
\end{theorem}

\begin{proof}
By rescaling we may assume $\e=1$.
We cover $z+Y_*$ by a finite number of balls $\{B(x_i, c)\}$, where either $x_i \in z+\partial T$ or $B(x_i, 3c ) \subset Y_*^+$.
In the case $x_i \in z+\partial T$, we use the estimates in Lemmas \ref{lem.Moser} and \ref{lem.SS.Lip}.
If $B(x_i, 3c)\subset Y_*^+$, then $\phi \approx 1$ in $B(x_i, 2c)$.
This allows us to apply the classical regularity theory.
\end{proof}

Finally, we prove the local Lipschitz estimate near the boundary.

\begin{theorem}\label{lem.Local Linfinity2}
Let $\Omega$ be a bounded $C^{1,\alpha}$ domain. Let $u_\e\in V_\e$ be a solution of $\mathcal{L}_\e (u_\e) = \phi_\e F $ in $\Omega_\e$.
    Then for $p>d$,
    \begin{equation}\label{local-b1}
    \sup_{B\cap \Omega_\e } |u_\e| \le C\bigg( \fint_{2B\cap \Omega_\e } |\phi_\e u_\e|^2 \bigg)^{1/2} +C\e^2
    \bigg( \fint_{2B\cap \Omega_\e } |F|^p \bigg)^{1/p}
\end{equation}
    and
    \begin{equation}\label{local-b2}
        \sup_{B\cap \Omega_\e } |\phi_\e \nabla u_\e| 
        \le C\bigg( \fint_{2B\cap \Omega_\e } |\phi_\e \nabla u_\e|^2 \bigg)^{1/2} 
        +C\e \bigg( \fint_{2B\cap \Omega_\e } |F|^p \bigg)^{1/p},
    \end{equation}
    where $B=B(x_0, C \e)$ for any $x_0 \in \partial \Omega$.
\end{theorem}

\begin{proof}
We cover $B(x_0, C \e)\cap \Omega_\e$ by a finite number of balls $\{ B(x_i, c\e)\}$ such that the balls may be divided into three
cases: (1) $x_i \in \partial \Omega_\e\setminus \partial \Omega$, (2) $B(x_i, 3c\e) \subset \Omega_\e$, and (3) $x_i \in \partial \Omega$.
The first two cases have already been discussed in the proof of the last theorem.
If dist$(x, \partial \Omega)\le c\e $, then dist$(\partial \Omega, \partial T_\e) \ge c_0 \e$ and thus $\phi_\e (x)\approx  1 $.
As a result, the desired estimate for the third case
follows readily from the classical boundary Lipschitz estimates. 
\end{proof}

\begin{remark}\label{rmk.Lip.T'}
    If $\Omega$ is only a Lipschitz domain, then we do not expect the boundary Lipschitz estimates in Theorem \ref{lem.Local Linfinity2}. However, since the boundary $\partial \Omega$ does not touch the holes, we can still get the Lipschitz estimate in a neighborhood of the holes $T$. Precisely, let $T'$ be the extended holes given by $T'=\{ x\in Y_*: \dist(x,T) \le c_0/8 \}$. Then due to Lemma \ref{lem.SS.Lip},  even if $\e(z+Y_*) \cap \partial \Omega \neq \emptyset$, we still have
    \begin{equation}
        \sup_{\e(z+T'\setminus T)}  |\phi_\e \nabla u_\e| 
        \le C\bigg( \fint_{\e(z+ Y^+_*) \cap \Omega} |\phi_\e \nabla u_\e|^2 \bigg)^{1/2} 
        +C\e \bigg( \fint_{\e(z+ Y_*^+) \cap \Omega} |F|^p \bigg)^{1/p}.
    \end{equation}
\end{remark}

\section{Quantitative homogenization}\label{sec.4}

In this section, we will show the quantitative convergence rate (i.e., the first-order approximation) for the degenerate equation in a perforated domain, 
\begin{equation}\label{eq.ue}
    \cL_\e(u_\e) = \phi_\e f \quad \text{in } \Omega_\e \quad \text{and}\quad  u_\e = g \quad \text{on } \partial \Omega,
\end{equation}
where $f\in L^p(\Omega_\e)$ with $p>d$ and $g\in H^1(\Omega)$. An algebraic convergence rate will be essential in proving the large-scale uniform regularity.

We call $u_\e$ a weak solution of \eqref{eq.ue} if $u_\e-g \in V_\e$ 
and
$$
\int_{\Omega_\e} \phi_\e^2 A_\e \nabla u_\e \cdot \nabla \psi =\int_{\Omega_\e} \phi_\e f \psi
$$
for any $\psi \in V_\e$. Since $H^1(\Omega) \subset W^{1,2}_{\phi_\e}(\Omega_\e)$, it is not hard to see that \eqref{eq.ue} has a unique solution in $V_\e$ satisfying
\begin{equation}\label{est.energy}
    \| \phi_\e \nabla u_\e \|_{L^2(\Omega_\e)} \le C\big( \| f\|_{L^2(\Omega_\e)} + \| \phi_\e \nabla g\|_{L^2(\Omega_\e)} \big).
\end{equation}

\subsection{Correctors and the homogenized equation}

Let
\begin{equation*}
    V=\left\{ v\in L^1_{\rm loc} (Y_*):
    \phi v\in L^2(Y_*), \ \phi\nabla v\in  L^2(Y_*)^d, \int_{Y_*} v=0 \text{ and  $v$ is $Y$-periodic} \right\}
\end{equation*}
be a Hilbert space with norm $\| v\|_V = \| \phi \nabla v \|_{L^2(Y_*)}$.
Let $\chi_j \in V$ be the weak solution of
\begin{equation*}
    -\text{\rm div} (\phi^2 A \nabla \chi_j ) =\text{\rm div} (\phi^2  A\nabla y_j) \quad \text{ in } Y_*,
\end{equation*}
i.e. for any $v\in V$, 
\begin{equation*}
    \int_{Y_*} \phi^2 A \nabla \chi_j \cdot \nabla v
    =-\int_{Y_*} \phi^2 A \nabla y_j \cdot \nabla v.
\end{equation*}
The vector-valued function $\chi=(\chi_j)$ is called the first-order correctors for the
operator $\mathcal{L}_\e =-\text{\rm div} (\phi_\e^2 A_\e \nabla )$.
Observe  that 
$$
\text{\rm div} (\phi^2 A \nabla (\chi_j + y_j))=0 \quad \text{ in }
\mathbb{R}^d \setminus \bigcup_{z\in \mathbb{Z}^d} (z+T).
$$
It follows from Theorem \ref{lem.Local Linfinity} that $|\chi| \in L^\infty(Y_*)$ and
$|\phi \nabla \chi| \in L^\infty(Y_*)$.

The homogenized coefficient matrix $\widehat{A}= (\widehat{a}_{ij} )$ is defined by
\begin{equation*}
    \widehat{a}_{ij}
    =\int_{Y_*}
    \phi^2 a_{ik} \partial_k \left( y_j +\chi_j\right) dy,
\end{equation*}
where $1\le i, j \le d$ and the repeated index $k$ is summed from $1$ to $d$. Note that 
$$
\widehat{a}_{ij}
=\int_{Y_*}
\phi^2 A \nabla (y_j+\chi_j) \cdot \nabla (y_i+\chi_i)\, dy.
$$
Since $A$ is positive definite and $\phi>0$ in $Y_*$, it is not hard to see that the constant matrix $\widehat{A}$  is positive
definite.
Under the assumption that $A$ is symmetric, $\widehat{A}$ is also symmetric.
The homogenized operator for $\mathcal{L}_\e$ is given by
\begin{equation*}
    \mathcal{L}_0 =-\text{\rm div} ( \widehat{A} \nabla ).
\end{equation*}

\subsection{First-order approximation}\label{sec.1stOrder}

In this subsection, we derive the error estimate of the first-order approximation for the boundary value problem.
We will show that an effective approximate problem for \eqref{eq.ue} is given by
\begin{equation}\label{eq.u00}
    \cL_0(u_0) = F_\e \quad \text{in } \Omega, \quad \text{and}\quad  u_0 = g \quad \text{on } \partial \Omega,
\end{equation}
where $F_\e\in L^p(\Omega)$ is the extension of $\phi_\e f$ by zero from $\Omega_\e$ to $\Omega$.
Throughout this subsection, we only assume that $\Omega$ is a Lispchitz domain satisfying the geometric assumption \eqref{H}.  


Let $c_1 \in (0,\frac14 c_0]$ be a constant. Let $\eta_\e \in C_0^\infty(\Omega)$ be a cutoff function such that $\eta_\e = 1$ if $\dist(x,\partial \Omega)>2c_1 \e$, $\eta_\e = 0$ if $\dist(x,\partial \Omega)< c_1 \e$, and $|\nabla\eta_\e| \le C\e^{-1}$. Define
\begin{equation*}
    \Omega(t\e) = \left\{ x\in \Omega: \dist(x,\partial \Omega) < tc_1 \e \right\}.
\end{equation*}
Observe that $\nabla \eta_\e$ is supported in a thin layer $\Omega(2\e) \setminus \Omega(\e) = \{ x\in \Omega: c_1 \e \le \dist(x,\partial \Omega) < 2c_1 \e \}$. 
Let
\begin{equation}\label{def.we}
    w_\e = u_\e - u_0 - \e \chi_\ell(x/\e) (\partial_\ell u_0) \eta_\e.
\end{equation}
Since $ F_\e \in L^2(\Omega)$, $u_0\in H_{\rm loc}^2(\Omega)$. Recall that $\chi_\ell$ and $\phi \nabla \chi_\ell$ are both bounded. 


\begin{lemma}\label{lem.Dwe.Dh}
    For any $\psi \in H^1_0(\Omega)$, we have
    \begin{equation}\label{est.DweDh}
    \aligned
        & \bigg| \int_{\Omega_\e} \phi_\e^2 A_\e \nabla w_\e \cdot \nabla \psi \bigg|\\
        & \le C \e \| \nabla^2 u_0 \|_{L^2(\Omega\setminus \Omega(\e))} \| \nabla \psi \|_{L^2(\Omega )} + C \| \nabla u_0 \|_{L^2(\Omega(2\e))} \| \nabla \psi \|_{L^2(\Omega(2\e))} .
    \endaligned
    \end{equation}
\end{lemma}

\begin{proof}
First of all,  $\psi \in H^1_0(\Omega)$ implies $\psi \in V_\e $. Thus $\psi$ can be used as a test function.
We calculate directly
\begin{equation}\label{eq.Dwe.Dh}
\begin{aligned}
    \int_{\Omega_\e} \phi_\e^2 A_\e \nabla w_\e \cdot \nabla \psi & = \int_{\Omega_\e} \phi_\e^2 A_\e\nabla u_\e \cdot \nabla \psi
    - \int_{\Omega_\e} \phi_\e^2 A_\e \nabla u_0 \cdot \nabla \psi \\
    &\qquad  - \int_{\Omega_\e} \phi_\e^2 A_\e (\nabla \chi_\ell)_\e (\partial_\ell u_0) \eta_\e \cdot \nabla \psi \\
    &\qquad- \e \int_{\Omega_\e} \phi_\e^2 A_\e(\chi_\ell)_\e (\nabla \partial_\ell u_0) \eta_\e \cdot \nabla \psi \\
    & \qquad - \e \int_{\Omega_\e} \phi_\e^2 A_\e (\chi_\ell)_\e \partial_\ell u_0 \nabla \eta_\e \cdot \nabla \psi .
\end{aligned}
\end{equation}
Using the equations \eqref{eq.ue} and \eqref{eq.u00}, we have
\begin{equation*}
    \int_{\Omega_\e} \phi_\e^2 A_\e \nabla u_\e \cdot \nabla \psi  = \int_{\Omega} \widehat{A}\nabla u_0 \cdot \nabla \psi . 
\end{equation*}
Inserting this equation into \eqref{eq.Dwe.Dh}, we obtain
\begin{equation}\label{eq.Dwe.Dh2}
\begin{aligned}
    \int_{\Omega_\e} \phi_\e^2 A_\e \nabla w_\e \cdot \nabla \psi  & = \int_{\Omega} (\widehat{A} - \phi_\e^2 A_\e 
    - \phi_\e^2 A_\e (\nabla \chi)_\e ) (\nabla u_0) \eta_\e \cdot \nabla \psi  \\
    &\qquad + \int_{\Omega} (\widehat{A} - \phi_\e^2 A_\e ) \nabla u_0 (1-\eta_\e) \cdot \nabla \psi \\
    &\qquad  - \e \int_{\Omega_\e} \phi_\e^2 A_\e(\chi_\ell)_\e (\nabla \partial_\ell u_0) \eta_\e \cdot \nabla \psi \\
&\qquad     - \e \int_{\Omega_\e} \phi_\e^2 A_\e (\chi_\ell)_\e \partial_\ell u_0 \nabla \eta_\e \cdot \nabla \psi .
\end{aligned}
\end{equation}

Next, we define the flux correctors. Let $B = \widehat{A} - \phi^2 A - \phi^2 A\nabla \chi$, or in component form
\begin{equation*}
    b_{ij} = \widehat{a}_{ij} - \phi^2 a_{ij} - \phi^2 a_{ik} \partial_k \chi_j.
\end{equation*}
Observe that $b_{ij} \in  L^\infty(Y)$ , and
\begin{equation*}
    \int_{Y} b_{ij} = 0, \qquad \partial_i b_{ij} = 0.
\end{equation*}
Then it is well-known (see, e.g., \cite{Shen18}) that we can find the flux correctors $\Phi = (\Phi_{kij}) \in H_{\rm per}^1(Y)$ such that
\begin{equation*}
    b_{ij} = \partial_k \Phi_{kij}, \qquad \Phi_{kij} = -\Phi_{ikj}.
\end{equation*}
Moreover, $\Phi_{kij} \in C^\alpha(Y)$ for any $\alpha \in (0,1)$. 
Therefore, by using the skew-symmetry of $\Phi$ and integration by parts,
\begin{equation*}
\begin{aligned}
    & \int_{\Omega} (\widehat{A} - \phi_\e^2 A_\e - \phi_\e^2 A_\e (\nabla \chi)_\e ) (\nabla u_0) \eta_\e \cdot \nabla \psi \\
    & = \int_{\Omega} \partial_k(\e \Phi_{kij}(x/\e) ) (\partial_j u_0)( \partial_i \psi ) \eta_\e \\
    & = - \e \int_\Omega \Phi_{kij}(x/\e)(\partial_k\partial_j u_0)( \partial_i \psi ) \eta_\e -\e \int_{\Omega} \Phi_{kij}(x/\e) \partial_j u_0 \partial_i \psi  \partial_k \eta_\e.
\end{aligned}
\end{equation*}
The first integral on the right-hand side is bounded by $C \e \| \nabla^2 u_0 \|_{L^2(\Omega\setminus \Omega(\e))} \| \nabla \psi  \|_{L^2(\Omega )}$,
while the second integral is bounded by
\begin{equation}\label{est.Du0.layer}
    C \e \int_{\Omega(2\e)} |\nabla u_0| |\nabla \psi | \le C \e \| \nabla u_0 \|_{L^2(\Omega(2\e))} \| \nabla \psi  \|_{L^2(\Omega(2\e))}.
\end{equation}
Finally, note that the third integral on the right-hand side of \eqref{eq.Dwe.Dh2} is also bounded by $C \e \| \nabla^2 u_0 \|_{L^2(\Omega\setminus \Omega(\e))} \| \nabla h \|_{L^2(\Omega )}$, and the second and fourth integrals are bounded by the right-hand side of \eqref{est.Du0.layer}. Summing up all these estimates, we obtain the desired estimate.
\end{proof}


Now we pick a particular test function $\psi $. We point out that $w_\e$ itself, even with a cut-off on the boundary is not a good choice for $\psi $ since $\nabla u_\e$ may not lie in $L^2(\Omega)^d$ (we only know $\phi_\e \nabla u_\e \in L^2(\Omega)^d$). We will use the harmonic extension to handle the possible singularity of $\nabla u_\e$ near the holes. 
Let $\delta\in (0,c_1]$. Define
\begin{equation*}
    \Omega^\delta_\e: = \left\{x\in \Omega_\e: \dist(x,T_\e)>\delta \e\right \}.
\end{equation*}
Note that $|\Omega_\e \setminus \Omega^\delta_\e| \le C\delta$. We shall 
extend the function $u_\e$ and $\chi(x/\e)$ from $\Omega^\delta_\e$ to $\Omega$ with a suitable choice of $\delta$. We define $T^\delta = \{ x\in Y: \dist(x,T) \le \delta \}$ and $Y_*^\delta = Y\setminus T^\delta$.





Another technical tool we need is the smoothing operator. Let $0\le \zeta \in C_0^\infty(B_{c_1}(0))$ and $\int_{B_{c_1}(0)} \zeta = 1$ and define the standard smoothing operator by
\begin{equation*}
    \mathscr{K}_\e f(x) = \int_{\mathbb{R}^d}  \e^{-d} \zeta(\frac{x-y}{\e})f(y) dy.
\end{equation*}
Many properties about the above smoothing operators can be found in \cite[Chapter 3.1]{Shen18}. We include some useful properties in Appendix for the reader's convenience.

\begin{lemma}\label{lem.Dwet}
Let $\Omega_\e^\delta$ be given as above. Let $\tilde{u}_\e$ be the harmonic extension of $u_\e$ from $\Omega_\e^\delta$ to $\Omega_\e$. It holds
    \begin{equation}\label{est.FullRate}
        \begin{aligned}
            \int_{\Omega_\e} \phi_\e^2 |\nabla w_\e|^2 & \le C\e \delta^{-1} \| \nabla^2 u_0 \|_{L^2(\Omega\setminus \Omega(\e))} \big( \| \phi_\e \nabla u_\e \|_{L^2(\Omega_\e)} + \| \nabla u_0\|_{L^2(\Omega)} \big) \\
        & \qquad + C\| \nabla u_0 \|_{L^2(\Omega(2\e))}^2 + C\e^2 \| \nabla^2 u_0 \|_{L^2(\Omega\setminus \Omega(\e/2))}^2  + C\delta \| \nabla u_0\|_{L^2(\Omega)}^2 \\
        & \qquad + C\int_{\Omega_\e \setminus \Omega_\e^\delta} \phi_\e^2 (|\nabla u_\e|^2 + |\nabla\tilde{u}_\e|^2).
        \end{aligned}
    \end{equation}
\end{lemma}

\begin{proof}
     Let $\tilde{\chi}_\e = \tilde{\chi}(x/\e)$ denote 
     the harmonic extension of $\chi_\e = \chi(x/\e)$ from $\Omega_\e^\delta$ to $\Omega$. For simplicity, we sometimes also use the notations $(\chi_\ell)_\e = \chi_\ell(x/\e)$ and $(\nabla \chi)_\e = (\nabla \chi)(x/\e)$ (similar notations apply for $\tilde{\chi}$ and $\nabla \tilde{\chi}$).
     
     Let $\eta_\e$ be given as before. Define
    \begin{equation*}
        w_\e^* = u_\e - u_0 - \e \chi_\ell(x/\e) \mathscr{K}_\e(\partial_\ell u_0) \eta_\e,
    \end{equation*}
    and
    \begin{equation*}
    w^\delta_\e = \tilde{u}_\e  - u_0 -\e \tilde{\chi}_\ell(x/\e) \mathscr{K}_\e(\partial_\ell u_0) \eta_\e.
    \end{equation*}
    It is easy to verify that $w^\delta_\e \in H^1_0(\Omega)$. Write
    \begin{equation}\label{eq.Dwe.I123}
    \begin{aligned}
        & \int_{\Omega_\e} \phi_\e^2 A_\e \nabla w_\e \cdot \nabla w_\e \\
        & = \int_{\Omega_\e} \phi_\e^2 A_\e  \nabla w_\e \cdot \nabla w_\e^\delta + \int_{\Omega_\e} \phi_\e^2  A_\e\nabla w_\e \cdot \nabla (w_\e - w_\e^*) + \int_{\Omega_\e} \phi_\e^2 A_\e \nabla w_\e \cdot \nabla (w_\e^* - w_\e^\delta) \\
        & = I_1 + I_2 + I_3.        
    \end{aligned}
    \end{equation}

    \textbf{Estimate of $I_1$:} Since $w^\delta_\e \in H^1_0(\Omega)$, we apply Lemma \ref{lem.Dwe.Dh} to obtain
    \begin{equation}\label{est.I1}
        |I_1| \le C \e \| \nabla^2 u_0 \|_{L^2(\Omega\setminus \Omega(\e))} \| \nabla w^\delta_\e \|_{L^2(\Omega )} + C \| \nabla u_0 \|_{L^2(\Omega(2\e))} \| \nabla w^\delta_\e \|_{L^2(\Omega(2\e))} .
    \end{equation}
    By the triangle inequality,
    \begin{equation}\label{est.Dwedelta}
    \begin{aligned}
        & \| \nabla w^\delta_\e \|_{L^2(\Omega)} \\
        & \le \| \nabla \tilde{u}_\e \|_{L^2(\Omega)} + \| \nabla u_0 \|_{L^2(\Omega)} + \e \| \nabla (\tilde{\chi}_\ell(x/\e) \mathscr{K}_\e(\partial_\ell u_0) \eta_\e) \|_{L^2(\Omega)} \\
        & \le \| \nabla \tilde{u}_\e \|_{L^2(\Omega)} + C\| \nabla u_0 \|_{L^2(\Omega)} + C\| (\nabla\tilde{\chi})_\e \mathscr{K}_\e(\nabla u_0) \|_{L^2(\Omega\setminus \Omega(\e))} \\
        &\qquad\qquad\qquad
        + C\e\| \nabla^2 u_0 \|_{L^2(\Omega\setminus \Omega(\e/2))}.
    \end{aligned}
    \end{equation}
    The boundedness of the harmonic extension (see Lemma \ref{lem.Gradient Ext}) implies
    \begin{equation*}
        \| \nabla \tilde{u}_\e \|_{L^2(\Omega)} \le C\| \nabla u_\e \|_{L^2(\Omega_\e^\delta)} \le C\delta^{-1} \| \phi_\e \nabla u_\e \|_{L^2(\Omega_\e)},
    \end{equation*}
    where we used the fact $\phi_\e \ge c\delta$ on $\Omega_\e^\delta$. The boundedness of $\phi \nabla \chi$  implies 
    \begin{equation*}
        \| \nabla \tilde{\chi} \|_{L^2(Y)} \le C \| \nabla {\chi} \|_{L^2(Y^\delta_*)} \le C\delta^{-1},
    \end{equation*}
    where $Y^\delta_* = \{x\in Y: \dist(x,T) >\delta \}$.
    Hence, we can use a property of the smoothing operator (see Lemma \ref{lem.smoothing}) to obtain
    the estimate of the third term in the right-hand side  of \eqref{est.Dwedelta}, 
    \begin{equation*}
        \| (\nabla\tilde{\chi})_\e \mathscr{K}_\e(\nabla u_0) \|_{L^2(\Omega\setminus \Omega(\e))} \le C \| \nabla \tilde{\chi} \|_{L^2(Y)} \| \nabla u_0 \|_{L^2(\Omega)}  \le C\delta^{-1} \| \nabla u_0 \|_{L^2(\Omega)}.
    \end{equation*}
 
    To handle the second term in \eqref{est.I1}, we note that $\Omega(2\e)$ is contained in $\Omega^\delta_\e$, provided that $\delta \in (0,\frac14 c_1)$. Thus $w^\delta_\e = w^*_\e$ and $\phi_\e \approx 1$ in $\Omega(2\e)$. Consequently,
    \begin{equation*}
    \begin{aligned}
        & C\| \nabla u_0\|_{L^2(\Omega(2\e))} \| \nabla w^\delta_\e \|_{L^2(\Omega(2\e))} \\
        & \le C\| \nabla u_0 \|_{L^2(\Omega(2\e))}^2 + \frac18 \| \phi_\e \nabla w_\e^* \|_{L^2(\Omega(2\e))}^2 \\
        & \le C\| \nabla u_0 \|_{L^2(\Omega(2\e))}^2 + \frac14 \| \phi_\e \nabla w_\e \|_{L^2(\Omega_\e)}^2 + \| \phi_\e \nabla (w^*_\e -w_\e) \|_{L^2(\Omega(2\e))}^2.
    \end{aligned}
    \end{equation*}
    Observe that $w_\e^* - w_\e = \e \chi_\ell(x/\e) (\partial_\ell u_0 - \mathscr{K}_\e(\partial_\ell u_0)) \eta_\e$ and
    \begin{equation*}
    \begin{aligned}
        \phi_\e \nabla (w_\e^* - w_\e) & = \phi_\e (\nabla \chi_\ell)_\e (\partial_\ell u_0 - \mathscr{K}_\e(\partial_\ell u_0)) \eta_\e \\    
        & \qquad + \e \phi_\e (\chi_\ell)_\e (\nabla \partial_\ell u_0 - \mathscr{K}_\e(\nabla \partial_\ell u_0)) \eta_\e \\
        & \qquad +  \e \phi_\e\nabla \eta_\e (\chi_\ell)_\e (\partial_\ell u_0 - \mathscr{K}_\e(\partial_\ell u_0)).
    \end{aligned}
    \end{equation*}
    Now, by using the boundedness of $\chi$ and $\phi \nabla \chi$, as well as the properties of the smoothing operator $\mathscr{K}_\e$ (see Lemma \ref{lem.smoothing} and Lemma \ref{lem.smoothing.error}), we see that 
    \begin{equation}\label{est.Dw*-Dw}
        \| \phi_\e \nabla (w^*_\e -w_\e) \|_{L^2(\Omega_\e)} \le C\e \| \nabla^2 u_0 \|_{L^2(\Omega\setminus \Omega(\e/2))}.
    \end{equation}
    As a result of \eqref{est.I1} - \eqref{est.Dw*-Dw}, we obtain
    \begin{equation*}
    \begin{aligned}
        |I_1| & \le C\e \delta^{-1} \| \nabla^2 u_0 \|_{L^2(\Omega\setminus \Omega(\e))} \big( \| \phi_\e \nabla u_\e \|_{L^2(\Omega_\e)} + \| \nabla u_0\|_{L^2(\Omega)} \big) \\
        & \qquad + C\| \nabla u_0 \|_{L^2(\Omega(2\e))}^2 + C\e^2 \| \nabla^2 u_0 \|_{L^2(\Omega\setminus \Omega(\e/2))}^2+  \frac14 \| \phi_\e \nabla w_\e \|_{L^2(\Omega_\e)}^2.
    \end{aligned}
    \end{equation*}

    \textbf{Estimate of $I_2$:} It follows from \eqref{est.Dw*-Dw} and the Cauchy-Schwarz inequality that
    \begin{equation*}
        |I_2| \le \frac14 \| \phi_\e \nabla w_\e \|_{L^2(\Omega_\e)}^2 + C\e^2 \| \nabla^2 u_0 \|_{L^2(\Omega\setminus \Omega(\e/2))}^2.
    \end{equation*}

    \textbf{Estimate of $I_3$:} Again, it follows from the Cauchy-Schwarz inequality that
    \begin{equation}\label{est.I3}
        |I_3| \le \frac14 \| \phi_\e \nabla w_\e \|_{L^2(\Omega_\e)}^2 + 4 \| \phi_\e \nabla (w^*_\e -w_\e^\delta) \|_{L^2(\Omega_\e)}^2.
    \end{equation}
    Recall that $w_\e^* = w_\e^\delta$ in $\Omega_\e^\delta$. Thus $\phi_\e \nabla (w^*_\e -w_\e^\delta)$ is supported in $\Omega_\e \setminus \Omega_\e^\delta$. Moreover,
    \begin{equation}\label{est.I3-2}
    \begin{aligned}
        \phi_\e \nabla (w^*_\e -w_\e^\delta) & = \phi_\e \nabla u_\e - \phi_\e \nabla \tilde{u}_\e \\
        & \qquad - \phi_\e \nabla (\e \chi_\ell(x/\e) \mathscr{K}_\e(\partial_\ell u_0) \eta_\e ) + \phi_\e \nabla (\e \tilde{\chi}_\ell(x/\e) \mathscr{K}_\e(\partial_\ell u_0) \eta_\e ).
    \end{aligned}
    \end{equation}
    We estimate the last term of the above identity over $\Omega_\e \setminus \Omega_\e^\delta$. Note that $\phi_\e \le C\delta$ and $\eta_\e = 1$ in $\Omega_\e \setminus \Omega_\e^\delta$. It follows from the triangle inequality and the boundedness of $\tilde{\chi}$,
    \begin{equation*}
    \begin{aligned}
        & \int_{\Omega_\e \setminus \Omega_\e^\delta} \phi_\e^2 |\nabla (\e \tilde{\chi}_\ell(x/\e) \mathscr{K}_\e(\partial_\ell u_0) \eta_\e )|^2 \\
        & \le C\delta^2 \e^2 \int_{\Omega_\e \setminus \Omega_\e^\delta} |\mathscr{K}_\e(\nabla^2 u_0)|^2  + C\delta^2 \int_{\Omega_\e \setminus \Omega_\e^\delta} |\nabla \tilde{\chi}(x/\e)|^2 |\mathscr{K}_\e(\nabla u_0)|^2 \\
        & \le C\delta^2 \e^2 \int_{\Omega_\e \setminus \Omega(\e)} |\nabla^2 u_0|^2 + C\delta^2 \| \nabla \tilde{\chi} \|_{L^2(Y_*\setminus Y^\delta_*)}^2 \int_{\Omega} |\nabla u_0|^2,
    \end{aligned}
    \end{equation*}
    where we have used Lemma \ref{lem.smoothing} in the last inequality. By the $L^2$ regularity estimate of the nontangential maximal function for the harmonic function $\tilde{\chi}$ in the Lipschitz holes $T^\delta = \{ x\in Y: \dist(x,T)<\delta \}$, we have (see Lemma \ref{lem.nontangential.L2} in Appendix)
    \begin{equation*}
        \| \nabla \tilde{\chi} \|_{L^2(Y_*\setminus Y^\delta_*)}^2  \le C\delta \| \nabla_{\rm \tan} \chi \|_{L^2(\partial T^\delta)}^2 \le C\delta^{-1}. 
    \end{equation*}
    Hence,
    \begin{equation*}
        \int_{\Omega_\e \setminus \Omega_\e^\delta} \phi_\e^2 |\nabla (\e \tilde{\chi}_\ell(x/\e) \mathscr{K}_\e(\partial_\ell u_0) \eta_\e )|^2 \le C\delta^2 \e^2 \int_{\Omega_\e \setminus \Omega(\e)} |\nabla^2 u_0|^2 + C\delta \int_{\Omega} |\nabla u_0|^2.
    \end{equation*}
    Similarly (and easier), using the boundedness of $\chi$ and $\phi \nabla \chi$, we have
    \begin{equation*}
        \int_{\Omega_\e \setminus \Omega_\e^\delta} \phi_\e^2 |\nabla (\e {\chi}_\ell(x/\e) \mathscr{K}_\e(\partial_\ell u_0) \eta_\e )|^2 \le C\delta^2 \e^2 \int_{\Omega_\e \setminus \Omega(\e)} |\nabla^2 u_0|^2 + C\delta \int_{\Omega} |\nabla u_0|^2.
    \end{equation*}
    Combining these with \eqref{est.I3} and \eqref{est.I3-2}, we obtain
    \begin{equation*}
    \begin{aligned}
        |I_3| & \le \frac14 \| \phi_\e \nabla w_\e \|_{L^2(\Omega_\e)}^2 + C\int_{\Omega_\e \setminus \Omega_\e^\delta} \phi_\e^2 (|\nabla u_\e|^2 + |\nabla\tilde{u}_\e|^2) \\
        & \qquad + C \e^2 \|\nabla^2 u_0\|_{L^2(\Omega_\e \setminus \Omega(\e))}^2 + C\delta \| \nabla u_0\|_{L^2(\Omega)}^2.
    \end{aligned}
    \end{equation*}

    Finally, the desired estimate follows from \eqref{eq.Dwe.I123} and the estimates of $I_1, I_2$ and $I_3$.
    \end{proof}

\subsection{Convergence rates}

In general, if $\Omega$ is a Lipschitz domain and $g\in H^1(\partial \Omega)$, we do not have $u_0\in H^2(\Omega)$, though we still have $u_0 \in H^2_{\rm loc}(\Omega)$ due to the interior $H^2$ estimate. However, we have the following estimates proved in \cite{Shen17} (a sketch of the proof is also given in Appendix).
\begin{lemma}\label{lem.u0inLip}
    Let $\Omega$ be a Lipschitz domain and let $u_0$ be the weak solution of \eqref{eq.u00} with $F = F_\e\in L^2(\Omega)$ and $g\in H^1(\partial \Omega)$. Then
    \begin{equation}\label{est.u0layer-2}
        \| \nabla^2 u_0 \|_{L^2(\Omega \setminus \Omega(c\e))} \le C\e^{-\frac12} \big( \| F \|_{L^2(\Omega)} + \| g \|_{H^1(\partial \Omega)} \big),
    \end{equation}
    and
    \begin{equation}\label{est.u0layer-1}
        \| \nabla u_0 \|_{L^2( \Omega(2\e))} \le C\e^{\frac12} \big( \| F \|_{L^2(\Omega)} + \| g \|_{H^1(\partial \Omega)} \big).
    \end{equation}
\end{lemma}

The above lemma indicates that the right-hand side of \eqref{est.FullRate} is small by choosing $\delta$ appropriately small, except for the last integral. The smallness of the last integral follows from the small-scale higher integrability of $\phi \nabla u_\e$ and $\phi_\e \nabla \tilde{u}_\e$. In particular, the small-scale Lipschitz estimate proved earlier ensures the smallness of the last integral and therefore a (suboptimal) convergence rate.

    \begin{lemma}\label{lem.deltaLayer}
    Let $\Omega$ be a Lipschitz domain. Let $f\in L^p(\Omega_\e)$ for some $p >d$ and $u_\e$ a solution of \eqref{eq.ue}. Then
    \begin{equation*}
        \int_{\Omega_\e \setminus \Omega^\delta_\e} \phi_\e^2 |\nabla u_\e|^2 \le C\delta \big( \| \phi_\e \nabla u_\e \|_{L^2(\Omega_\e)}^2 + \e^{2-\frac{2d}{p}} \| f \|_{L^p(\Omega_\e)}^2\big).
    \end{equation*}
\end{lemma}
\begin{proof}
    This lemma essentially relies only on the small-scale interior Lipschitz estimate, i.e., Theorem \ref{lem.Local Linfinity} and Remark \ref{rmk.Lip.T'}. Indeed, for $\delta<c_0/8$,
    \begin{equation*}
        \begin{aligned}
            \int_{\Omega_\e \setminus \Omega^\delta_\e} \phi_\e^2 |\nabla u_\e|^2 & = \sum_{\e(z+T) \subset \Omega } \int_{\e(z+T^\delta \setminus T)} \phi_\e^2 |\nabla u_\e|^2 \\
            & \le \sum_{\e(z+T)\subset \Omega} C\e^d \delta \sup_{\e(z+T^\delta \setminus T)} |\phi_\e \nabla u_\e|^2 \\
            & \le \sum_{\e(z+T)\subset \Omega } C\e^d \delta \bigg\{ \fint_{\Omega_\e \cap \e(z+Y_*^+)} |\phi_\e \nabla u_\e|^2 + \e^{2} \bigg( \fint_{\Omega_\e \cap \e(z+Y_*^+)} |f|^p \bigg)^{2/p} \bigg\} \\
            & \le \sum_{\e(z+T)\subset \Omega } C\delta \bigg\{ \int_{\Omega_\e \cap \e(z+Y_*^+)} |\phi_\e \nabla u_\e|^2 + \e^d \e^{2-\frac{2d}{p}} \| f \|_{L^p(\Omega_\e)}^2  \bigg\} \\
            & \le C\delta \big( \| \phi_\e \nabla u_\e \|_{L^2(\Omega_\e)}^2 + \e^{2-\frac{2d}{p}} \| f \|_{L^p(\Omega_\e)}^2\big).
        \end{aligned}
    \end{equation*}
    The proof is complete.
\end{proof}

\begin{lemma}\label{lem.deltaLayer-2}
     Let $u_\e$ be the same as in Lemma \ref{lem.deltaLayer} and
     $\tilde{u}_\e$  the harmonic extension of $u_\e$ from $\Omega_\e^\delta$ to $\Omega_\e$. Then for $\delta<c_0/8$,
     \begin{equation*}
         \int_{\Omega_\e \setminus \Omega^\delta_\e} \phi_\e^2 |\nabla \tilde{u}_\e|^2 \le C\delta \big( \| \phi_\e \nabla u_\e \|_{L^2(\Omega_\e)}^2 + \e^{2-\frac{2d}{p}} \| f\|_{L^p(\Omega_\e)}^2 \big).
     \end{equation*}
\end{lemma}
\begin{proof}
    First, we write
    \begin{equation}\label{est.Dtue.split}
        \int_{\Omega_\e \setminus \Omega^\delta_\e} \phi_\e^2 |\nabla \tilde{u}_\e|^2 = \sum_{\e(z+T ) \subset \Omega} \int_{\e(z+T^\delta \setminus T )} \phi_\e^2 |\nabla \tilde{u}_\e|^2.
    \end{equation}
    We then consider a single cell $\e(z+T^\delta \setminus T) \subset \Omega$. By Theorem \ref{lem.Local Linfinity} and Remark \ref{rmk.Lip.T'},
    \begin{equation}\label{est.Due.Linfty}
        \| \nabla u_\e \|_{L^\infty(\e(z+Y_* \setminus T^\delta))} \le C\delta^{-1}\bigg( \fint_{\e(z+ Y^+_*) \cap \Omega_\e} |\phi_\e \nabla u_\e|^2 \bigg)^{1/2} +C\e \delta^{-1} \bigg( \fint_{\e(z+ Y_*) \cap \Omega_\e} |f|^p \bigg)^{1/p}.
    \end{equation}
    Recall that $\Delta \tilde{u}_\e = 0$ in $\e(z + T^\delta)$ and $\tilde{u}_\e = u_\e$ on $\e(z + \partial T^\delta)$. Also, $T^\delta = \cup_{i} \tau_i^\delta$ is a union of pairwise disjoint Lipschitz holes with connected boundaries. 
    In each hole, we can apply the $L^2$ regularity estimate of the nontangential estimate (see \eqref{est.L2reg} in Appendix) to obtain 
    \begin{equation*}
    \begin{aligned}
        \| (\nabla \tilde{u}_\e)^* \|_{L^2(\e(z+\partial T^\delta))} &\le C \| \nabla_{\rm tan} \tilde{u}_\e \|_{L^2(\e(z+\partial T^\delta))} \\
        & \le C \| \nabla {u}_\e \|_{L^2(\e(z+\partial T^\delta))} \le C \e^{\frac{d-1}{2}} \| \nabla u_\e \|_{L^\infty(\e(z+Y_* \setminus T^\delta))}.
    \end{aligned}
    \end{equation*}
    Consequently,
    \begin{equation}\label{est.Tdelta.Dtue}
        \int_{\e(z+ T^\delta\setminus T)} \phi_\e^2 |\nabla \tilde{u}_\e|^2 \le 
        C\delta^3 \e \int_{\e(z+\partial T^\delta)} |(\nabla \tilde{u}_\e)^* |^2 \le C\delta^3 \e^d \| \nabla u_\e \|_{L^\infty(\e(z+Y_* \setminus T^\delta))}^2.
    \end{equation}
    Combining \eqref{est.Due.Linfty} and \eqref{est.Tdelta.Dtue}, we see that 
    \begin{equation*}
        \int_{\e(z+ T^\delta\setminus T)} \phi_\e^2 |\nabla \tilde{u}_\e|^2 \le C\delta \int_{\e(z+ Y^+_*) \cap \Omega_\e} |\phi_\e \nabla u_\e|^2 + C\delta \e^d \e^{2-\frac{2d}{p}} \| f \|_{L^p(\Omega_\e)}^2.
    \end{equation*}
    Summing over $z$, we obtain
    \begin{equation*}
        \int_{\Omega_\e \setminus \Omega^\delta_\e} \phi_\e^2 |\nabla \tilde{u}_\e|^2 \le C\delta \big( \| \phi_\e \nabla u_\e \|_{L^2(\Omega_\e)}^2 + \e^{2-\frac{2d}{p}} \| f\|_{L^p(\Omega_\e)}^2 \big),
    \end{equation*}
    as desired.
\end{proof}

\begin{theorem}\label{thm.L2rate}
    Let $\Omega$ be a Lipschitz domain, $f\in L^p(\Omega_\e)$ for some $p>d$ and $g\in H^1(\partial \Omega)$. Then
    \begin{equation*}
        \int_{\Omega_\e} \phi_\e^2 |\nabla w_\e|^2 \le C\e^{\frac14} \big( \|f \|_{L^p(\Omega_\e)}^2 + \| g \|_{H^1(\partial \Omega)}^2 \big).
    \end{equation*}
\end{theorem}

\begin{proof}
    First, by the energy estimate \eqref{est.energy} and Lemma \ref{lem.u0inLip}, we have
    \begin{equation*}
        \e^{\frac12} \| \nabla^2 u_0 \|_{L^2(\Omega(c\e))} + \e^{-\frac12} \| \nabla u_0 \|_{L^2(\Omega(c\e))} + \| u_\e \|_{H^1_{\phi_\e}(\Omega_\e)} \le C\big( \| f\|_{L^2(\Omega_\e)} + \| g \|_{H^1(\partial \Omega)} \big).
    \end{equation*}
    Hence, Lemma \ref{lem.Dwet} implies
    \begin{equation*}
    \begin{aligned}
        \int_{\Omega_\e} \phi_\e^2 |\nabla w_\e|^2 & \le C(\e^{\frac12} \delta^{-1} + \e + \delta ) \big( \|f \|_{L^p(\Omega_\e)}^2 + \| g \|_{H^1(\partial \Omega)}^2 \big) + C \int_{\Omega_\e \setminus \Omega^\delta_\e} \phi_\e^2 \big( |\nabla u_\e|^2 + |\nabla \tilde{u}_\e|^2 \big)\\
        & \le C(\e^{\frac12} \delta^{-1} + \e + \delta ) \big( \|f \|_{L^p(\Omega_\e)}^2 + \| g \|_{H^1(\partial \Omega)}^2 \big),
    \end{aligned}
    \end{equation*}
    where we have used Lemma \ref{lem.deltaLayer} and Lemma \ref{lem.deltaLayer-2} in the second inequality. Since $\delta$ is arbitrary, we pick $\delta = \e^{\frac14}$ to obtain the desired estimate.
\end{proof}

The following local convergence rate is particularly useful for us. Let $Q_r$ be a cube with side length $r$ and $Q_r^\e = Q_r \setminus T_\e$.

    \begin{theorem}\label{thm.local.rate}
        Let $u_\e\in H^1_{\phi_\e}(Q_{2r}^\e)$ be the weak solution of $\cL_\e (u_\e) = \phi_\e f$ in $Q_{2r}^\e$ with $r = m\e \ge \e$. There exists a weak solution $u_0\in H^1(Q_r)$ of the homogenized equation $\cL_0(u_0) = \phi_\e f$ in $Q_{\frac43 r}$ such that $w_\e$ satisfies
        \begin{equation}\label{G-rate}
            \fint_{Q_{r}^\e } \phi_\e^2 |\nabla w_\e|^2 \le C\Big( \frac{\e}{r} \Big)^{\frac14} \bigg\{ \fint_{Q_{2r}^\e } \phi_\e^2 |\nabla u_\e|^2 +   r^2 \bigg( \fint_{Q_{2r}^\e} |f|^p \bigg)^{2/p} \bigg\},
        \end{equation}
        where $w_\e = u_\e - u_0 - \e\chi(x/\e) \cdot \nabla u_0$ in $Q_r^\e$. Moreover,
        \begin{equation}\label{est.localL2rate}
            \fint_{Q_{r}^\e } \phi_\e^2 |u_\e - u_0 |^2 \le C\Big( \frac{\e}{r} \Big)^{\frac14} \bigg\{ \fint_{Q_{2r}^\e } \phi_\e^2 |u_\e|^2 +   r \bigg( \fint_{Q_{2r}^\e} |f|^p \bigg)^{2/p} \bigg\},
        \end{equation}        
        The estimates still hold if $Q_{2r}^\e$ is replaced by $Q_{2r}^\e(x_0) \cap \Omega_\e$ with $x_0\in \partial \Omega$ and $u_\e = 0$ on $\partial \Omega \cap Q_{2r}^\e(x_0)$.
    \end{theorem}

    \begin{proof}
    We give the proof for \eqref{G-rate}.
    The inequality \eqref{est.localL2rate} follows from \eqref{G-rate} by using a weighted Poincar\'e's inequality and Caccioppoli's inequality.
    By rescaling, it suffices to show \eqref{G-rate} for the case $r = 1$. Let $t\in (1,\frac32)$ be a suitable side length such that $Q_t^\e$ satisfies the geometric assumption \eqref{H} (in particular $\partial Q_t \subset Q_2^\e$) and
    \begin{equation}\label{est.H1.coarea}
        \| u_\e \|_{H^1(\partial Q_t)} \le C\| u_\e\|_{W^{1,2}_{\phi_\e}(Q^\e_{3/2})}.
    \end{equation}
    The choice of such $t$ is possible due to the assumptions on $T$ and the co-area formula.
    Let $u_0$ be the weak solution of
    \begin{equation}
        \left\{ \begin{aligned}
            \cL_0(u_0)  &= \phi_\e f, \quad \text{in } Q_t^\e,\\
            u_0 &= u_\e,\quad \text{on } \partial Q_t.
        \end{aligned}
        \right.
    \end{equation}
    Taking the estimates in Theorem \ref{thm.local.rate} into Lemma \ref{lem.Dwet} with $\Omega_\e = Q_t^\e$, we obtain
    \begin{equation}
    \begin{aligned}
        \int_{Q_t^\e} \phi_\e^2 |\nabla w_\e|^2 & \le C\e^\frac14 \big( \|f \|_{L^p(Q_2^\e)}^2 + \| u_\e \|_{H^1(\partial Q_t)}^2 \big) \\
        & \le C\e^\frac14 \big( \|f \|_{L^p(Q_2^\e)}^2 + \| u_\e\|_{W^{1,2}_{\phi_\e}(Q^\e_2)}^2 \big).
    \end{aligned}
    \end{equation}
    Here the cutoff function $\eta_\e$ in the original definition of $w_\e$ \eqref{def.we} has been removed by analyzing the boundary layers.
    Finally, notice that we can subtract any constant $L$ from $u_\e$ in the above estimate. Hence, $\| u_\e\|_{W^{1,2}_{\phi_\e}(Q_2^\e)}$ may be replaced by $\| \phi_\e \nabla u_\e\|_{L^2(Q^\e_2)}$ in view of \eqref{est.u-Lq.Dup}. This ends the proof.
\end{proof}

\begin{remark}
    The exponent $\frac14$ in Theorem \ref{thm.L2rate} and Theorem \ref{thm.local.rate} is not optimal . By adjusting the cutoff function $\eta_\e$, we can actually improve it to $\frac13$. If we know a priori that $u_0 \in H^2(\Omega)$ (for instance, $\Omega$ is convex or $C^{1,1}$ and $g = 0$), then it may be improved to $\frac12$. It is an interesting question whether the sharp convergence rate of $O(\e)$ holds under the condition $f\in L^p(\Omega_\e)$ or $\phi_\e \nabla f \in L^p(\Omega)^d$ (even with $g = 0$). Note that if $f$ and $\Omega$ are smooth enough, then the sharp convergence rate is valid by using the maximum principle; see \cite[Chapter III.3]{OSY92}.
\end{remark}

\section{Lipschitz estimates}\label{sec.5}


In Section \ref{sec.3}, we have proved the small-scale Lipschitz estimate.
In this section, we establish the large-scale Lipschitz estimate which leads to the full-scale Lipschitz estimate. Our argument here follows from the scheme of \cite{Shen17} (also see \cite{Shen18}).

We define
\begin{equation}
    H_\e(u;r) = \inf_{M\in \R^{d}, q \in \R} \frac{1}{r}\bigg( \fint_{Q_r^\e} \phi_\e^2 | u -M\cdot x-q|^2 \bigg)^{1/2} + r\bigg( \fint_{Q_r^\e} |f|^p \bigg)^{1/p},
\end{equation}
and
\begin{equation}
    I_\e(u;r) = \inf_{q \in \R} \frac{1}{r} \bigg( \fint_{Q_r^\e} \phi_\e^2 | u -q|^2 \bigg)^{1/2} + r\bigg( \fint_{Q_r^\e} |f|^p \bigg)^{1/p}.
\end{equation}
Recall that the interior Caccioppoli inequality for the weak solution of $\cL_\e(u_\e) = \phi_\e f$ can be written as
\begin{equation}\label{est.Caccioppoli}
    \bigg( \fint_{Q_r^\e} \phi_\e^2 |\nabla u_\e |^2 \bigg)^{1/2} \le CI_\e(u_\e;2r).
\end{equation}

\begin{lemma}\label{lem.Exc.Q2}
    Let $u_\e$ be a weak solution of $\cL_\e(u_\e) = \phi_\e f$ in $Q_2^\e$ with $f\in L^p(Q_2^\e)$ for some $p>d$. Then there exist $\theta \in (0,\frac12)$ and $\e_0>0$ such that if $\e<\e_0$,
    \begin{equation}\label{Est.Exc.1step}
        H_\e (u_\e; \theta ) \le \frac12  H_\e(u_\e; 1) +  C \e^\frac{1}{8} I(u_\e;2).
    \end{equation}
\end{lemma}
\begin{proof}
    Let $u_0$ be given as in the proof of Theorem \ref{thm.local.rate} with $r = 1$. Then
    \begin{equation}\label{est.ue-u0.inQ1}
        \bigg( \fint_{Q_{1}^\e } \phi_\e^2 |u_\e - u_0 |^2 \bigg)^{1/2} \le C\e^\frac18 I_\e(u_\e;2).
    \end{equation}
    Note that we have replaced $u_\e$ by $u_\e - L$ on the right-hand side of \eqref{est.localL2rate} and taken infimum over all $L\in \R$.
    On the other hand, the energy estimate and \eqref{est.H1.coarea} implie
    \begin{equation}
        \| \nabla u_0 \|_{L^2(Q_t)} \le C\| \phi_\e \nabla u_\e \|_{L^2(Q_{3/2}^\e)} + C\| f \|_{L^2(Q_{3/2}^\e)}.
    \end{equation}
    By the interior estimate of the equation $\cL_0(u_0) = \phi_\e f$ in $Q_1$,
    \begin{equation}\label{est.u0.C1a}
    \aligned
        \| u_0 \|_{C^{1,\alpha}(Q_{3/4})}  & \le C\| \phi_\e \nabla u_\e \|_{L^2(Q_{3/2}^\e)} + C\| f \|_{L^p(Q_{3/2}^\e)}\\
        & \le CI(u_\e;2),
        \endaligned
    \end{equation}
    where $\alpha = 1-\frac{d}{p}$ and we have used the Caccioppoli estimate in the last line.
    
    Now let $\theta \in (0,1/2)$. By the triangle inequality and \eqref{est.ue-u0.inQ1}
    \begin{equation}\label{est.Exc.theta}
    \begin{aligned}
        H_\e(u_\e; \theta ) & \le H_\e(u_0;\theta) + \frac{1}{\theta} \bigg( \fint_{Q^\e_\theta} |u_\e - u_0|^2 \bigg)^{1/2}\\
        & \le H_\e(u_0;\theta) + \frac{C}{\theta^{1+d/2}} \e^\frac18 I_\e(u_\e;2).
    \end{aligned}
    \end{equation}

    Next, we claim that for any $\theta \in (0,1/2)$ (independent of $\e$), 
    \begin{equation}\label{est.claim.Hu0}
        H_\e(u_0;\theta) \le C\theta^\alpha H_\e(u_0;3/4) + C\e I_\e (u_0;2).
    \end{equation}
    In fact, for $\theta \in (0,1/2)$, by observing that $\cL_0(u_0-M\cdot x -q) = \cL_0(u_0) = \phi_\e f$ in $Q_1$ for any $M\in \R^d$ and $q\in \R$, we have
    \begin{equation}\label{est.He.theta2}
    \begin{aligned}
        H_\e(u_0;\theta) & \le \theta^\alpha \| \nabla u_0 \|_{C^\alpha(Q_{1/2})} + \theta^{1-d/p} \| F \|_{L^p(Q^\e_\theta)} \\
        & \le C \theta^\alpha \bigg\{ \bigg( \fint_{Q_{3/4}} |u_0 - M_0 \cdot x-q_0 |^2 \bigg)^{1/2} + \| f \|_{L^p(Q^\e_{3/4})} \bigg\},
    \end{aligned}
    \end{equation}
    where we choose $M_0\in \R^d$ and $q_0 \in \R$ such that
    \begin{equation}\label{def.M0q0}
        \bigg( \fint_{Q_{3/4}} \phi_\e^2 |u_0 - M_0 \cdot x-q_0 |^2 \bigg)^{1/2} = \inf_{M\in \R^d,q\in \R} \bigg( \fint_{Q_{3/4}} \phi_\e^2 |u_0 - M \cdot x-q |^2 \bigg)^{1/2}.
    \end{equation}
    Let $a_0 = \int_{Y_*} \phi^2>0$. Then it is easy to see that
    \begin{equation*}
    \aligned
        \sqrt{a_0}\bigg( \fint_{Q_{3/4}} |M_0 \cdot x+q_0 |^2 \bigg)^{1/2}  & \le C \bigg( \fint_{Q_{3/4}} \phi_\e^2 |M_0 \cdot x+q_0 |^2 \bigg)^{1/2}\\
        &\le C\bigg( \fint_{Q_{3/4}} \phi_\e^2 |u_0 |^2 \bigg)^{1/2}.
        \endaligned
    \end{equation*}
    This further implies that
    \begin{equation}\label{est.M0q0}
        |M_0| + |q_0| \le C\bigg( \fint_{Q_{3/4}} \phi_\e^2 |u_0 |^2 \bigg)^{1/2} \le C\| u_0 \|_{L^\infty(Q_{3/4})}.
    \end{equation}
    Moreover,
    there exists a bounded $Y$-periodic vector-valued function $\psi$ such that $\nabla\cdot (\e \psi_\e) = \phi_\e^2 - a_0$.
    It follows from the integration by parts, \eqref{est.M0q0} and \eqref{est.u0.C1a} that
    \begin{equation*}
    \begin{aligned}
        & \bigg| \bigg( \fint_{Q_{3/4}}(\phi_\e^2-a_0) |u_0 - M_0 \cdot x-q_0 |^2 \bigg)^{1/2} \bigg|\\
        & \le C \e \bigg( \fint_{\partial Q_{3/4}} |\psi_\e\cdot n|^2 |u_0 - M_0 \cdot x-q_0|^2 \bigg)^{1/2} \\
        &\qquad\qquad+ C\e \bigg( \fint_{Q_{3/4}} |\psi_\e||u_0 - M_0 \cdot x-q_0 | | \nabla u_0 - M_0| \bigg)^{1/2} \\
        & \le C\e \| u_0 \|_{C^{0,1}(Q_{3/4})} \\
        &  \le C\e I_\e(u_\e;2).
    \end{aligned}
    \end{equation*}
    This, combined with \eqref{est.He.theta2} and \eqref{def.M0q0},  gives
    \begin{equation*}
        H(u_0;\theta) \le C\theta^\alpha H(u_0;3/4) + C\e I(u_\e;2).
    \end{equation*}
    Thus, the claim \eqref{est.claim.Hu0} is proved.
    
    Finally, by \eqref{est.Exc.theta}, \eqref{est.claim.Hu0} and \eqref{est.ue-u0.inQ1}, we arrive at
    \begin{equation*}
    \aligned
        H_\e(u_\e;\theta)  & \le C\theta^\alpha H_\e(u_\e;3/4) + C \theta^{-1-\frac{d}{2}} \e^\frac18 I(u_\e;2) \\
        & \le C\theta^\alpha H_\e(u_\e;1) + C \theta^{-1-\frac{d}{2}} \e^\frac18 I(u_\e;2).
        \endaligned
    \end{equation*}
    Note that the constant $C$ above is independent of $\theta$. 
    As a result, we can choose and fix $\theta \in (0,1/2)$ such that $C\theta^\alpha \le 1/2$. This gives \eqref{Est.Exc.1step} as desired.
\end{proof}

We recall an iteration lemma.
\begin{lemma}[{\cite[Lemma 6.4.6]{Shen18}}] \label{Lem.Iteration}
    Let $H(r)$ and $h(r)$ be two nonnegative, continuous functions on the interval $(0,1]$. Let $0<\e<1/4$. Suppose that there exists a constant $C_0$ such that
    \begin{equation}\label{cond.sizeHh}
        \max_{r\le t\le 2r} H(t) \le C_0 H(2r) \quad \text{and} \quad \max_{r\le s, t\le 2r} |h(s) - h(t)| \le C_0 H(2r)
    \end{equation}
    for any $r\in [\e,1/2]$. Suppose further that
    \begin{equation}\label{cond.Hr}
        H(\theta r) \le \frac12 H(r) + C_0 \beta(\e/r) (H(2r) + h(2r)),
    \end{equation}
    for any $r\in [\e,1/2]$, where $\theta \in (0,1/4)$ and $\beta(t)$ is a nonnegative, nondecreasing function on $[0,1]$ such that $\beta(0) = 0$ and
    \begin{equation}\label{cond.Dini}
        \int_0^1 \frac{\beta(t)}{t} dt < \infty.
    \end{equation}
    Then
    \begin{equation}
        \max_{\e\le r\le 1} (H(r) + h(r) ) \le C  (H(1) + h(1) ),
    \end{equation}
    where $C$ depends only on $C_0,\theta$ and the function $\beta(t)$.
\end{lemma}

\begin{theorem}\label{thm.LargeScaleLip}
    Under the same condition as Lemma \ref{lem.Exc.Q2}, for $\e \le r\le  2$, we have 
    \begin{equation}\label{est.LargeScaleLip}
        \bigg( \fint_{Q_{r}^\e} \phi_\e^2 |\nabla u_\e|^2 \bigg)^{1/2} \le C \bigg( \fint_{Q_{2}^\e} \phi_\e^2 |\nabla u_\e|^2 \bigg)^{1/2} + C\bigg( \fint_{Q_2^\e} |f|^p \bigg)^{1/p}.
    \end{equation}
\end{theorem}
\begin{proof}
    By rescaling, \eqref{Est.Exc.1step} holds for any $\e/\e_0 \le r\le 1$, i.e.,
    \begin{equation}\label{est.Hr}
        H_\e( u_\e; \theta r) \le \frac{1}{2} H_\e(u_\e; r) + C_0 \Big( \frac{\e}{r} \Big)^\frac18 I_\e(u_\e; 2r).
    \end{equation}
    To apply Lemma \ref{Lem.Iteration}, we set $H(t) = H_\e(u_\e;t)$. In view of the definition of $H_\e(u_\e;t)$, we can find $M_t \in \R^d$ such that
    \begin{equation*}
        H(t) = H_\e(u_\e;t) = \inf_{q\in \R} \frac{1}{t}\bigg( \fint_{Q_t^\e} \phi_\e^2 | u -M_t\cdot x-q|^2 \bigg)^{1/2} + t\bigg( \fint_{Q_t^\e} |f|^p \bigg)^{1/p}.
    \end{equation*}
    Define $h(t) = |M_t|$. Then it is obvious that
    \begin{equation}\label{est.Ie2Hh}
        I_\e(u_\e;t) \le H(t) + C h(t).
    \end{equation}
    Hence, we can rewrite \eqref{est.Hr} as, for $\e/\e_0 \le r\le 1/2$,
    \begin{equation*}
        H(\theta r) \le \frac12 H(r) + C_0 \Big( \frac{\e}{r} \Big)^\frac18 ( H(2r) + h(2r)),
    \end{equation*}
    which verifies the main condition \eqref{cond.Hr} in Lemma \ref{Lem.Iteration} with $\beta(t) = t^\frac18$. Clearly, this particular $\beta(t)$ satisfies \eqref{cond.Dini}.

    Finally, to see \eqref{cond.sizeHh}, we only need to use the properties of $H(t) = H_\e(u_\e;t)$ and $h(t) = |M_t|$. In fact, for any $r\le t\le 2r$, we have $H(t) \le C H(2r)$ by enlarging the region from $Q^\e_t$ to $Q^\e_{2r}$ with comparable volumes. Moreover, for any $r\le s,t\le 2r$,
    \begin{equation*}
    \begin{aligned}
        & |h(s) - h(t)|  \le |M_s - M_t| \le \frac{C}{r}\bigg( \fint_{Q_r^\e} \phi_\e^2 | (M_s-M_t) \cdot x |^2 \bigg)^{1/2} \\
        &\le C\inf_{q\in \R} \frac{1}{r}\bigg( \fint_{Q_r^\e} \phi_\e^2 | u_\e - M_s\cdot x -q |^2 \bigg)^{1/2} + C \inf_{q\in \R} \frac{1}{r}\bigg( \fint_{Q_r^\e} \phi_\e^2 | u_\e - M_t\cdot x -q |^2 \bigg)^{1/2} \\
        & \le C \inf_{q\in \R} \frac{1}{r}\bigg( \fint_{Q_s^\e} \phi_\e^2 | u_\e - M_s\cdot x -q |^2 \bigg)^{1/2} + C\inf_{q\in \R} \frac{1}{r}\bigg( \fint_{Q_t^\e} \phi_\e^2 | u_\e - M_t\cdot x -q |^2 \bigg)^{1/2} \\
        & \le CH(s) + CH(t) \le C_0 H(2r).
    \end{aligned}
    \end{equation*}
    This verifies the condition \eqref{cond.sizeHh}. It follows from Lemma \ref{Lem.Iteration} that
    \begin{equation*}
        \max_{\e/\e_0 \le r\le 1} (H(r) + h(r) ) \le C  (H(1) + h(1) ),
    \end{equation*}
    By the Caccioppli inequality \eqref{est.Caccioppoli} and \eqref{est.Ie2Hh}, for $\e/\e_0\le r\le 1$, we have
    \begin{equation*}
    \begin{aligned}
        \bigg( \fint_{Q^\e_r} |\nabla u_\e|^2 \bigg)^{1/2} & \le CI_\e(u_\e;2r) \le C(H(2r) + h(2r)) \\
        &\le C  (H(1) + h(1) ) \\
        & \le C \bigg( \fint_{Q_{2}^\e} \phi_\e^2 |\nabla u_\e|^2 \bigg)^{1/2} + C\bigg( \fint_{Q_2^\e} |f|^p \bigg)^{1/p},
    \end{aligned}
    \end{equation*}
    where in the last inequality we also used the weighted Poincar\'{e} inequality \eqref{est.u-Lq.Dup} and the size estimate of $h(1)$. Note that the estimate in the range $\e \le r\le \e/\e_0$ follows trivially from the case $r = \e/\e_0$ with a larger constant $C$. This completes  the proof.
\end{proof}

\begin{remark}\label{rmk.bdaryLip}
    The above estimates continue to hold near the boundary if $\Omega$ is a $C^{1,\alpha}$ domain and $u_\e = 0$ on $\partial \Omega$. In fact, if $Q_r^\e$ is centered on $\partial \Omega$, then we only need to replace $Q_r^\e$ by $Q_r^\e \cap \Omega$ in the statement of Theorem \ref{thm.LargeScaleLip}. For the proof, we  modify the function $H_\e$ and $I_\e$ as follows:
    \begin{equation*}
    \begin{aligned}
        H_\e(u;r) & = \inf_{M\in \R^{d}} \frac{1}{r} \bigg\{ \bigg( \fint_{Q_r^\e} \phi_\e^2 | u -M\cdot x|^2 \bigg)^{1/2} + \| M\cdot x \|_{L^\infty(Q_r\cap \partial \Omega)} \\
    & \qquad + r\| \nabla_{\rm tan} (M\cdot x) \|_{L^\infty(Q_r\cap \partial \Omega)}  + r^2 \bigg( \fint_{Q_r^\e} |f|^p \bigg)^{1/p} \bigg\},
    \end{aligned}
\end{equation*}
and
\begin{equation*}
    I_\e(u;r) = \frac{1}{r} \bigg( \fint_{Q_r^\e} \phi_\e^2 | u|^2 \bigg)^{1/2} + r\bigg( \fint_{Q_r^\e} |f|^p \bigg)^{1/p}.
\end{equation*}
    The details are omitted (see \cite{Shen17} for a similar proof).
\end{remark}


\begin{proof}[Proof of Theorem \ref{main-thm-1} (i)] Without loss of generality assume $\text{diam}(\Omega) \approx 1$.
First, by the small-scale Lipschitz estimate \eqref{local-b2}, we have
\begin{equation*}
    \sup_{Q_{\e}^\e \cap \Omega } |\phi_\e \nabla u_\e| \le C \bigg( \fint_{Q_{m_0 \e}^\e \cap \Omega} \phi_\e^2 |\nabla u_\e|^2 \bigg)^{1/2} + C\e \bigg( \fint_{Q_{m_0  \e}^\e \cap \Omega} |f|^p \bigg)^{1/p}.
\end{equation*}
Then by the large-scale Lipschitz estimate \eqref{est.LargeScaleLip} for $r = m_0 \e$ (including the boundary case in Remark \ref{rmk.bdaryLip}), we have
\begin{equation*}
    \bigg( \fint_{Q_{m_0 \e}^\e \cap \Omega} \phi_\e^2 |\nabla u_\e|^2 \bigg)^{1/2} \le C \bigg( \fint_{Q_{2}^\e \cap \Omega} \phi_\e^2 |\nabla u_\e|^2 \bigg)^{1/2} + C\bigg( \fint_{Q_2^\e \cap \Omega} |f|^p \bigg)^{1/p}.
\end{equation*}
Combining the last two estimates, we obtain
\begin{equation*}
    \sup_{Q_{\e}^\e \cap \Omega } |\phi_\e \nabla u_\e| \le C \bigg( \fint_{Q_{2}^\e \cap \Omega} \phi_\e^2 |\nabla u_\e|^2 \bigg)^{1/2} + C\bigg( \fint_{Q_2^\e \cap \Omega} |f|^p \bigg)^{1/p}.
\end{equation*}
Since the above estimate can be translated arbitrarily, it implies
\begin{equation}\label{est.Q1toQ2Lip}
    \sup_{Q_{1}^\e \cap \Omega } |\phi_\e \nabla u_\e| \le C \bigg( \fint_{Q_{2}^\e \cap \Omega} \phi_\e^2 |\nabla u_\e|^2 \bigg)^{1/2} + C\bigg( \fint_{Q_2^\e \cap \Omega} |f|^p \bigg)^{1/p},
\end{equation}
and
\begin{equation*}
    \| \phi_\e \nabla u_\e \|_{L^\infty(\Omega_\e)} \le C\| \phi_\e \nabla u_\e \|_{L^2(\Omega_\e)} + C\| f\|_{L^p(\Omega_\e)}.
\end{equation*}
This implies \eqref{est.GlobalLip} by the energy estimate \eqref{est.energy}.
\end{proof}

\section{$W^{1,p}$ estimates}\label{sec.6}

In this section, we prove the global $W^{1,p}$ estimates uniform in $\e$ in both Theorem \ref{main-thm-1} and Theorem \ref{main-thm-2}. Like Lipschitz estimate, the small-scale and large-scale estimates are handled  separately.

We begin with the small-scale $W^{1,p}$ estimate. In view of the discussion in Section \ref{sec.3}, it is sufficient to focus on the boundary case II.
\begin{lemma}\label{lem.SS.W1p}
    Let $p>2$ and $f\in L^p(Y_*)^d, F\in L^p(Y_*)$. Suppose that $u$ is a weak solution of
    \begin{equation*}
        -\text{\rm div}(\phi^2 A \nabla u) = \text{\rm div}(\phi f) + F, \quad \text{\rm in } Y_*.
    \end{equation*}
    Then if $B_0$ is a ball centered on $\partial T$ and $\text{\rm diam}(B_0) \le c_0/4$, then
    \begin{equation*}
        \bigg( \fint_{\frac12 B_0 \cap Y_*} |\phi \nabla u|^p \bigg)^{1/p} \le C\bigg\{ \bigg( \fint_{2B_0 \cap Y_*} |\phi \nabla u|^2 \bigg)^{1/2} + \bigg( \fint_{2B_0 \cap Y_*} (|f|^p + |F|^p ) \bigg)^{1/p} \bigg\}.
    \end{equation*}
\end{lemma}

The proof relies on the following interior real-variable argument.
\begin{theorem}[{\cite[Theorem 4.2.3]{Shen18}}]\label{thm.int.realvar}
        Let $q>2$ and $B_0$ be a ball in $\R^d$. Let $U\in L^2(4B_0)$ and $f\in L^p(4B_0)$ for some $2<p<q$. Suppose that for each ball or cube $B \subset 2B_0$ with properties that $\text{\rm diam}(B) \le r_0 \text{\rm diam}(B_0)$, there exist two measurable function $U_B$ and $R_B$ defined in $2B$, such that $|U| \le |U_B| + |R_B|$ in $2B$, and
    \begin{equation*}
        \bigg( \fint_{{2B}} |R_B|^q \bigg)^{1/q} \le N_1 \bigg\{ \bigg( \fint_{{4B}} |U|^2 \bigg)^{1/2} + \bigg( \fint_{{4B}} |f|^2 \bigg)^{1/2} \bigg\},
    \end{equation*}
    \begin{equation*}
        \bigg( \fint_{{2B}} |U_B|^2 \bigg)^{1/2} \le N_2  \bigg( \fint_{{4B}} |f|^2 \bigg)^{1/2} + \eta \bigg( \fint_{{4B}} |U|^2 \bigg)^{1/2},
    \end{equation*}
    where $N_1,N_2>1$ and $0<r_0<1$. Then there exists $\eta_0>0$, depending only on $N_1,N_2,r_0,p$ and $q$, with the property that if $0\le \eta < \eta_0$, then $U\in L^p(B_0)$ and
    \begin{equation}\label{est.LS.realvar-2}
        \bigg( \fint_{B_0} |U|^p dx \bigg)^{1/p} \le C \bigg\{ \bigg( \fint_{4B_0} |U|^2 \bigg)^{1/2} + \bigg( \fint_{4B_0} |f|^p \bigg)^{1/p} \bigg\}.
    \end{equation}
    where $C$ depends at most on $N_1,N_2,r_0,p$ and $q$.
\end{theorem}

\begin{proof}[Proof of Lemma \ref{lem.SS.W1p}]
    Without loss of generality, assume $\text{diam}(B_0) = c_0/10$. Let $B \subset B_0$ be a ball such that either $4B \subset Y_*$ or $B$ is centered on $\partial T$. For the first case, $\phi(x) \approx \phi (y) \approx  \text{diam}(B)$ for any $x, y \in 2B$. Let $v$ be the weak solution of
    \begin{equation*}
        -\text{div}(\phi^2 A \nabla v) = \text{\rm div}(\phi f) + F \quad \text{in } 2B, \quad \text{ and } \quad v = 0 \quad \text{on } \partial (2B).
    \end{equation*}
    By the energy estimate, we have
    \begin{equation}\label{est.RealVar-1}
        \| \phi \nabla v \|_{L^2(2B)} \le C\big( \| f \|_{L^2(2B)} + \| F \|_{L^2(2B)} \big).
    \end{equation}
    Let $w = u - v$. Then $w$ satisfies
    \begin{equation*}
        -\text{div}(\phi^2 A \nabla w) = 0 \quad \text{in } 2B, \quad \text{ and } \quad w = u \quad \text{on } \partial (2B).
    \end{equation*}
    Due to the fact $\phi 
    (x) \approx \phi (y) $ for $x, y\in 2B$, the classical Lipschitz estimate yields
    \begin{equation}\label{est.RealVar-2}
    \begin{aligned}
        \sup_{B} |\phi \nabla w| & \le C\bigg( \fint_{2B} |\phi \nabla w|^2 \bigg)^{1/2} \\
        & \le C\bigg( \fint_{2B} |\phi \nabla u|^2 \bigg)^{1/2} + C\bigg( \fint_{2B} (|f| + |F|)^2 \bigg)^{1/2},
    \end{aligned}  
    \end{equation}
    where we have used \eqref{est.RealVar-1} in the last inequality.

    For the second case that $B$ is centered on $\partial T$, we again let $v$ be the weak solution of
    \begin{equation*}
        -\text{div}(\phi^2 A \nabla v) = \text{\rm div}(\phi f) + F \quad \text{in } 2B\cap Y_*, \quad \text{ and } \quad v = 0 \quad \text{on } \partial (2B)\cap Y_*.
    \end{equation*}
    The energy estimate implies
    \begin{equation}\label{est.RealVar-3}
        \| \phi \nabla v \|_{L^2(2B\cap Y_*)} \le C\big( \| f \|_{L^2(2B\cap Y_*)} + \| F \|_{L^2(2B\cap Y_*)} \big).
    \end{equation}
    Then $w = u -v$ satisfies
    \begin{equation*}
        -\text{div}(\phi^2 A \nabla w) = 0 \quad \text{in } 2B\cap Y_*, \quad \text{ and } \quad w = u \quad \text{on } \partial (2B)\cap Y_*.
    \end{equation*}
    Now, we use \eqref{lem.SS.Lip} to get
    \begin{equation}\label{est.RealVar-4}
        \begin{aligned}
        \sup_{B\cap Y_*} |\phi \nabla w| & \le C\bigg( \fint_{2B \cap Y_*} |\phi \nabla w|^2 \bigg)^{1/2} \\
        & \le C\bigg( \fint_{2B\cap Y_*} |\phi \nabla u|^2 \bigg)^{1/2} + C\bigg( \fint_{2B\cap Y_*} (|f| + |F|)^2 \bigg)^{1/2},
    \end{aligned} 
    \end{equation}
    where we have used \eqref{est.RealVar-3} in the last inequality.

    Now, we explain how the interior real-variable argument is applied in $B_0$ to $U = \phi \nabla u$. Let $U$ be extended across $\partial T$ to $2B_0$ by zero. For each $B\subset B_0$, $B$ can be covered by a finite number of balls $B'_j$ (with finite overlaps) satisfying either one of the following three cases: (i) $4B'_j \subset Y_* $; (ii) $B'_j$ is centered on $\partial T$; (iii) $4B \cap Y_* = \emptyset$. Moreover, $\text{diam}(B'_j) \approx \text{diam}(B)$. In all three cases, we have $|U| \le U_j + R_j$ in $2B'_j$ with $U_j = |\phi \nabla v|$ and $R_j = |\phi \nabla w|$ supported in $2B'_j$, where $v$ and $w$ have been constructed previously for cases (i) and (ii). The case (iii) is trivial since $U = 0$ in $2B_0 \setminus Y_*$. Therefore, by \eqref{est.RealVar-1}, \eqref{est.RealVar-2}, \eqref{est.RealVar-3} and \eqref{est.RealVar-4}, we have
    \begin{equation*}
        \sup_{B'_j} |R_j| \le C\bigg( \fint_{2B'_j} |U|^2 \bigg)^{1/2} + C\bigg( \fint_{2B'_j} (|f| + |F|)^2 \bigg)^{1/2},
    \end{equation*}
    and
    \begin{equation*}
        \bigg( \fint_{2B'_j} |U_B|^2 \bigg)^{1/2} \le C\bigg( \fint_{2B'_j} (|f| + |F|)^2 \bigg)^{1/2},
    \end{equation*}
    where $F$ and $f$ are also extended by zero. Now, using $B \subset \cup_j B'_j$ and setting $R_B = \sum_{j} R_j$ and $U_B = \sum_{j} U_j$, we have $|U| \le R_B + U_B$ and 
    \begin{equation*}
        \sup_{B} |R_B| \le C\bigg( \fint_{2B} |U|^2 \bigg)^{1/2} + C\bigg( \fint_{2B} (|f| + |F|)^2 \bigg)^{1/2},
    \end{equation*}
    and
    \begin{equation*}
        \bigg( \fint_{B} |U_B|^2 \bigg)^{1/2} \le C\bigg( \fint_{2B} (|f| + |F|)^2 \bigg)^{1/2}.
    \end{equation*}
    Since $B\subset B_0$ is arbitrary, we apply Theorem \ref{thm.int.realvar} with $\eta = 0$ to obtain
    \begin{equation*}
        \bigg( \fint_{\frac12 B_0} |U|^p \bigg)^{1/p} \le C\bigg( \fint_{2B_0} |U|^2 \bigg)^{1/2} +  C\bigg( \fint_{2B_0} (|f| + |F|)^p \bigg)^{1/p}.
    \end{equation*}
    Since $U = \phi \nabla u$ and $f,F$ are extended by zero in the holes, the last inequality implies the desired estimate.
\end{proof}

    \begin{lemma}\label{lem.SmallScale.W1p}
        Let $p\ge 2$ and $f\in L^p(\Omega_\e)^d, F\in L^p(\Omega_\e)$. Let $u_\e\in H^1_{\phi_\e,0}(\Omega_\e)$ be the weak solution of \eqref{main.eq.F} in $\Omega_\e$. 
        Suppose $\e(k+Y_*) \cap \Omega_\e \neq \emptyset$. Then,
        \begin{equation*}
        \begin{aligned}
            \bigg( \fint_{\e(k+Y_*) \cap \Omega_\e} |\phi_\e \nabla u_\e|^p \bigg)^{1/p} & \le C\bigg( \fint_{\e(k+Y_*^+) \cap \Omega_\e} |\phi_\e \nabla u_\e|^2 \bigg)^{1/2} \\
            & \quad + C \bigg( \fint_{\e(k+Y_*^+) \cap \Omega_\e} |f|^p \bigg)^{1/p} + C\e \bigg( \fint_{\e(k+Y_*^+) \cap \Omega_\e} |F|^p \bigg)^{1/p}.
        \end{aligned}
        \end{equation*}
    \end{lemma}
    \begin{proof}
        We first rescale the problem. Let $v(x) = u_\e(\e x)$. Then $v$ satisfies
        \begin{equation*}
            -\text{div}(\phi^2 A \nabla v) = \text{\rm div}(\e \phi f(\e\cdot)) + \e^2 F(\e\cdot) \quad \text{in } (k+Y_{*}^+) \cap \e^{-1}\Omega_\e.
        \end{equation*}
        Moreover, $v = 0$ on $(k+Y_{*}^+) \cap \e^{-1} \partial \Omega$. Then, by Lemma \ref{lem.SS.W1p} and the classical $W^{1,p}$ estimate away from $\Gamma_\e$, we have
        \begin{equation*}
        \begin{aligned}
            \bigg( \fint_{(k+Y_*) \cap \e^{-1}\Omega_\e} |\phi \nabla v|^p \bigg)^{1/p} & \le C\bigg( \fint_{(k+Y_*^+) \cap \e^{-1}\Omega_\e} |\phi \nabla v|^2 \bigg)^{1/2} \\
&\qquad\qquad\qquad             + C\e \bigg( \fint_{(k+Y_*^+) \cap \e^{-1}\Omega_\e} |f(\e\cdot )|^p \bigg)^{1/p} \\
             &\qquad\qquad\qquad
             + C\e^2 \bigg( \fint_{\e(k+Y_*^+) \cap \Omega_\e} |F(\e\cdot)|^p \bigg)^{1/p}.
        \end{aligned}
        \end{equation*}
        Rescaling back to $u_\e$, we obtain the desired estimate.
    \end{proof}


To prove the global $W^{1,p}$ estimate uniform in $\e$ up to  $\partial \Omega$, we need a large-scale global real-variable argument.
\begin{theorem}[A variant of {\cite[Theorem 4.2.6]{Shen18}}]\label{thm.LS.real}
    Let $q>2$ and $\Omega$ be a bounded Lipschitz domain. Let $U\in L^2(\Omega)$ and $f\in L^p(\Omega)$ for some $2<p<q$. Let $0\le t < r_0 \text{\rm diam}(\Omega)$ be a fixed number. Suppose that for each ball or cube $B$ with properties that $t< \text{\rm diam}(B) \le r_0 \text{\rm diam}(\Omega)$ and either $4B\subset \Omega$ or $B$ is centered on $\partial \Omega$, there exist two measurable function $U_B$ and $R_B$ on $\Omega \cap 2B$, such that $|U| \le |U_B| + |R_B|$ on $\Omega \cap 2B$, and
    \begin{equation}\label{cond.LSRV-1}
        \bigg( \fint_{\Omega \cap 2B} |R_B|^q \bigg)^{1/q} \le N_1 \bigg\{ \bigg( \fint_{\Omega \cap 4B} |U|^2 \bigg)^{1/2} + \bigg( \fint_{\Omega \cap 4B} |f|^2 \bigg)^{1/2} \bigg\},
    \end{equation}
    \begin{equation}\label{cond.LSRV-2}
        \bigg( \fint_{\Omega \cap 2B} |U_B|^2 \bigg)^{1/2} \le N_2  \bigg( \fint_{\Omega \cap 4B} |f|^2 \bigg)^{1/2} + \eta \bigg( \fint_{\Omega \cap 4B} |U|^2 \bigg)^{1/2},
    \end{equation}
    where $N_1,N_2>1$ and $0<r_0<1$. Then there exists $\eta_0>0$, depending only on $N_1,N_2,r_0,p,q,$ and the Lipschtiz character of $\partial \Omega$, with the property that if $0\le \eta < \eta_0$, then $U\in L^p(\Omega)$ and
    \begin{equation}\label{est.LS.realvar}
        \bigg( \fint_{\Omega} \bigg( \fint_{B_t(x) \cap \Omega} |U|^2 \bigg)^{p/2} dx \bigg)^{1/p} \le C \bigg\{ \bigg( \fint_{\Omega} |U|^2 \bigg)^{1/2} + \bigg( \fint_{\Omega} |f|^p \bigg)^{1/p} \bigg\},
    \end{equation}
    where $C$ depends at most on $N_1,N_2,r_0,p,q,$ and the Lipschitz character of $\partial \Omega$. In particular, if $t = 0$, then \eqref{est.LS.realvar} is replaced by
    \begin{equation}\label{est.LS.realvar-3}
        \bigg( \fint_{\Omega} |U|^p dx \bigg)^{1/p} \le C \bigg\{ \bigg( \fint_{\Omega} |U|^2 \bigg)^{1/2} + \bigg( \fint_{\Omega} |f|^p \bigg)^{1/p} \bigg\}.
    \end{equation}
\end{theorem}

\begin{theorem}\label{thm.LS.W1p}
    Assume that $\Omega$ is $C^1$ or convex. Let $p>2$ and $f\in L^p(\Omega_\e)^d$ and $F\in L^p(\Omega_\e)$. Let $u_\e$ be the weak solution of \eqref{main.eq.F}. Then there exits $m_0>0$ such that for any $m_0 \e \le r\le \text{\rm diam}(\Omega)$, we have
    \begin{equation}
        \bigg( \int_{\Omega} \bigg( \fint_{Q_{r}(x) \cap \Omega_\e} |\phi_\e \nabla u_\e|^2 \bigg)^{p/2} dx \bigg)^{1/p} \le C \big( \| f \|_{L^p(\Omega_\e)} + \| F\|_{L^p(\Omega_\e)} \big).
    \end{equation}
\end{theorem}
\begin{proof}
We will apply Theorem \ref{thm.LS.real} to $U = \phi_\e \nabla u_\e$ with balls replaced by cubes. Let $Q_r^\e(x) = Q_r(x) \setminus T_\e$ with $r = m\e$ for some integer $m\ge m_0$. To make sure that $Q_r^\e(x)$ satisfies the geometric assumption \eqref{H}, we temporarily assume that $x$ is the center of some cell $\e(k+Y)$. Moreover, we will consider $Q_r^\e(x)$ such that either $Q_{4r}^\e(x) \subset \Omega_\e$ or $Q_r^\e(x)$ is centered in a boundary cell (i.e., a cell intersecting with $\partial \Omega$).
Consider
\begin{equation}\label{eq.ue.0}
    \cL_\e(u_\e) = \text{div}(\phi_\e f) + F \quad \text{in } Q_{2r}^\e \cap \Omega, \quad \text{and} \quad u_\e = 0 \quad \text{on } Q_{2r}^\e \cap \partial \Omega.
\end{equation}
Let $v_\e$ be the solution of 
\begin{equation*}
    \cL_\e(v_\e) = 0 \quad \text{in } Q_{2r}^\e \cap \Omega, \quad \text{and} \quad v_\e = u_\e \quad \text{on } \partial (Q_{2r}\cap \Omega).
\end{equation*}
Then $w_\e = u_\e - v_\e$ satisfies
\begin{equation*}
    \cL_\e(w_\e) = \text{div}(\phi_\e f) + F \quad \text{in } Q_{2r}^\e \cap \Omega, \quad \text{and} \quad w_\e = 0 \quad \text{on } \partial (Q_{2r}\cap \Omega).
\end{equation*}
Recall that $Q_{2r}^\e \cap \Omega$ satisfies the geometric assumption \eqref{H}. Thus, the energy estimate implies
\begin{equation}\label{est.we.fF}
    \bigg( \fint_{Q_{2r}^\e \cap \Omega} |\phi_\e\nabla w_\e|^2 \bigg)^{1/2} \le C\bigg( \fint_{Q_{2r}^\e \cap \Omega} |f|^2 \bigg)^{1/2} + C\bigg( \fint_{Q_{2r}^\e \cap \Omega} |F|^2 \bigg)^{1/2}.
\end{equation}

Next, we consider the regularity of $v_\e$. We distinguish between two cases. If $Q_{4r}^\e$ is contained in $\Omega$, then we have already proved the local uniform Lipschitz estimate of $v_\e$ in \eqref{est.Q1toQ2Lip} for $r>m_0 \e$, i.e.,
\begin{equation}\label{est.ve.Lip}
    \| \phi_\e \nabla v_\e \|_{L^\infty(Q_{r}^\e)} \le C\bigg( \fint_{Q_{2r}^\e} |\phi_\e\nabla v_\e|^2 \bigg)^{1/2}.
\end{equation}
In this case, let $U = \phi_\e \nabla u_\e$, $R_Q = \phi_\e \nabla v_\e$ and $U_Q = U-R_Q$. Then by \eqref{est.we.fF} and \eqref{est.ve.Lip}, we have
\begin{equation*}
    \| R_Q \|_{L^\infty(Q^\e_r)} \le C\bigg( \fint_{Q_{2r}^\e} |U|^2 \bigg)^{1/2} + C\bigg( \fint_{Q_{2r}^\e} (|f| + |F|)^2 \bigg)^{1/2},
\end{equation*}
for $r> m_0 \e$ and
\begin{equation*}
    \bigg( \fint_{Q_r^\e} |U_Q|^2 \bigg)^{1/2} \le C\bigg( \fint_{Q_{2r}^\e} (|f| + |F|)^2 \bigg)^{1/2}.
\end{equation*}
Previously we have assumed $r = m\e$ with some integer $m$ and $Q_{r}^\e$ is centered in a cell $\e(k+Y_*)$. But these assumptions can be removed by slightly adjusting the size of the cubes. Also the holes can be removed simply by a natural zero-extension. Hence, the conditions \eqref{cond.LSRV-1} and \eqref{cond.LSRV-2} are satisfied with $\eta = 0$, provided $Q_{4r} \subset \Omega$.

If $Q_{2r}^\e$ is centered in a boundary cell, we apply Theorem \ref{thm.local.rate} to approximate $v_\e$ with good functions such that the real-variable argument can apply. Precisely, by Theorem \ref{thm.local.rate}, there exists $v_0$ satisfying
\begin{equation*}
    \cL_0(v_0) = 0 \quad \text{in } Q_{\frac43 r}\cap \Omega \quad \text{and} \quad v_0 = 0 \quad \text{on } Q_{\frac43 r}\cap \partial \Omega,
\end{equation*}
and
\begin{equation}\label{est.ve-v0}
    \fint_{Q^\e_r\cap \Omega } \phi_\e^2 |\nabla v_\e - \nabla v_0 - \nabla\chi(x/\e)\cdot \nabla v_0|^2 \le C\Big( \frac{\e}{r} \Big)^{\frac14}  \fint_{Q_{2r}^\e\cap \Omega } \phi_\e^2 |\nabla v_\e|^2.
\end{equation}
Since $\partial \Omega$ is $C^1$ or convex, for any $q<\infty$, $\nabla v_0 \in L^q(Q_r \cap \Omega)$. Hence, by the boundedness of $\phi \nabla \chi$, we obtain for some $q>p$,
\begin{equation*}
    \bigg( \fint_{Q_r^\e\cap \Omega} \phi_\e^q|\nabla v_0 + \nabla \chi(x/\e) \nabla v_0|^q \bigg)^{1/q} \le C\bigg( \fint_{Q_{2r}^\e\cap \Omega } \phi_\e^2 |\nabla v_\e|^2 \bigg)^{1/2}.
\end{equation*}
Thus, if we set $U = \phi_\e \nabla u_\e$, $R_Q = \phi_\e \nabla v_0 + \phi_\e \nabla\chi(x/\e)\cdot \nabla v_0$ and $U_Q = U - R_Q$, then
\begin{equation*}
\begin{aligned}
    \bigg( \fint_{Q_r^\e \cap \Omega} |R_Q|^q \bigg)^{1/q} & \le C\bigg( \fint_{Q_{2r}^\e\cap \Omega } |U|^2 \bigg)^{1/2} + C\bigg( \fint_{Q_{2r}^\e \cap \Omega} (|f| + |F|)^2 \bigg)^{1/2},
\end{aligned}
\end{equation*}
and
\begin{equation*}
\begin{aligned}
    \bigg( \fint_{Q_r^\e \cap \Omega} |U_Q| \bigg)^{1/2} & \le C\Big( \frac{\e}{r} \Big)^{\frac18} \bigg( \fint_{Q_{2r}^\e\cap \Omega } |U|^2 \bigg)^{1/2} + C\bigg( \fint_{Q_{2r}^\e \cap \Omega} (|f| + |F|)^2 \bigg)^{1/2},
\end{aligned}
\end{equation*} 
Let $m_0>0$ be large enough such that
\begin{equation*}
    Cm_0^{-\frac18} < \eta_0,
\end{equation*}
where $\eta_0$ is given in Theorem \ref{thm.LS.real} depending on $C,p,q$ and $\Omega$. Hence, the conditions \eqref{cond.LSRV-1} and \eqref{cond.LSRV-2}  are also satisfied in this boundary case provided $r> m_0 \e$. As a consequence of Theorem \ref{thm.LS.real}, we obtain
\begin{equation*}
\begin{aligned}
    & \bigg( \int_{\Omega} \bigg( \fint_{Q_r^\e(x) \cap \Omega} |\phi_\e \nabla u_\e|^2 \bigg)^{p/2} dx \bigg)^{1/p} \\
    & \le C \bigg\{ |\Omega|^{1/p-1/2} \bigg( \int_{\Omega} |\phi_\e \nabla u_\e|^2 \bigg)^{1/2} + \bigg( \int_{\Omega} (|f| + |F|)^p \bigg)^{1/p} \bigg\} \\
    & \le C\| f\|_{L^p(\Omega_\e)} + C\| F\|_{L^p(\Omega_\e)}.
\end{aligned}
\end{equation*}
for all $r>m_0\e$. This proves the theorem.
\end{proof}

\begin{proof}[Proof of Theorem \ref{main-thm-2}]
    Note that $p=2$ corresponds to exactly the energy estimate. Let $p>2$. Then Theorem \ref{thm.LS.W1p} with $r = m_0 \e$ implies
    \begin{equation}\label{est.r=m0e}
        \bigg( \int_{\Omega} \bigg( \fint_{Q_{m_0\e}(x) \cap \Omega_\e} |\phi_\e \nabla u_\e|^2 \bigg)^{p/2} dx \bigg)^{1/p} \le C \big( \| f \|_{L^p(\Omega_\e)} + \| F\|_{L^p(\Omega_\e)} \big).
    \end{equation}
    Now Lemma \ref{lem.SmallScale.W1p} implies for any $x\in \Omega$,
    \begin{equation*}
    \begin{aligned}
        \bigg( \fint_{Q_{\e}(x) \cap \Omega_\e} |\phi_\e \nabla u_\e|^p \bigg)^{1/p} & \le C\bigg( \fint_{Q_{m_0\e}(x) \cap \Omega_\e} |\phi_\e \nabla u_\e|^2 \bigg)^{1/2} \\
        & \qquad + C\bigg( \fint_{Q_{m_0\e}(x) \cap \Omega_\e} |f|^p \bigg)^{1/p} + C\bigg( \fint_{Q_{m_0\e}(x) \cap \Omega_\e} |F|^p \bigg)^{1/p}.
    \end{aligned}
    \end{equation*}
    Substituting the last inequality into \eqref{est.r=m0e}, we obtain
    \begin{equation}\label{est.LS.ffLp}
    \begin{aligned}
        & \bigg( \int_{\Omega} \bigg[ \fint_{Q_{\e}(x) \cap \Omega_\e} |\phi_\e \nabla u_\e|^p \bigg] dx \bigg)^{1/p} \\
        & \le C \big( \| f \|_{L^p(\Omega_\e)} + \| F\|_{L^p(\Omega_\e)} \big) \\
        & \quad + C \bigg( \int_{\Omega} \bigg[ \fint_{Q_{m_0\e}(x) \cap \Omega_\e} |f|^p \bigg] dx \bigg)^{1/p} + C \bigg( \int_{\Omega} \bigg[ \fint_{Q_{m_0\e}(x) \cap \Omega_\e} |F|^p \bigg] dx \bigg)^{1/p}.
    \end{aligned}
    \end{equation}
    It is not hard to see that for any $g\in L^p(\Omega_\e)$, we have
    \begin{equation*}
        \int_{\Omega} \bigg[ \fint_{Q_{\e}(x) \cap \Omega_\e} |g|^p \bigg] dx \approx \int_{\Omega} \bigg[ \fint_{Q_{m_0 \e}(x) \cap \Omega_\e} |g|^p \bigg] dx \approx \int_{\Omega_\e} |g|^p.
    \end{equation*}
    This and \eqref{est.LS.ffLp} imply the desired estimate for $p>2$.

    Finally, the case $1<p<2$ follows from the case $p>2$ by duality.
\end{proof}

\begin{proof}[Proof of Theorem \ref{main-thm-1} (ii)]
This is a dual statement of Theorem \ref{main-thm-2}. Let $u_\e$ be the weak solution of \eqref{main.eq.phief} with $f\in L^{p}(\Omega_\e)$ and $1<p<d$. Let $g\in L^{p'}(\Omega)^d$ and $v_\e$ be the weak solution of
\begin{equation}\label{eq.ve-g}
    -\text{\rm div} (\phi^2_\e A_\e \nabla v_\e) = \text{\rm div} (\phi_\e g) \quad \text{ in } \Omega_\e\quad \text{ and }
\quad v_\e =0 \quad \text{ on } \partial \Omega.
\end{equation}
Then, by the equation \eqref{main.eq.phief} and \eqref{eq.ve-g}, we have
\begin{equation*}
    \int_{\Omega_\e} \phi_\e \nabla u_\e \cdot g = \int_{\Omega} f \phi_\e v_\e.
\end{equation*}
Theorem \ref{main-thm-2} and Theorem \ref{thm.WSPI} (i) imply
\begin{equation*}
\begin{aligned}
    \bigg| \int_{\Omega_\e} \phi_\e \nabla u_\e \cdot g \bigg| & = \| f\|_{L^p(\Omega_\e)} \| \phi_\e v_\e \|_{L^{p'}(\Omega_\e)} \\
    & \le C\| f\|_{L^p(\Omega_\e)} \| \phi_\e \nabla v_\e \|_{L^{q}(\Omega_\e)} \le C\| f\|_{L^p(\Omega_\e)} \| g \|_{L^{q}(\Omega_\e)},
\end{aligned}
\end{equation*}
where $p$ and $q$ satisfy $p' = q^*$, i.e., $1-\frac{1}{p} = \frac{1}{q} - \frac{1}{d}$. By duality, we have
\begin{equation*}
    \| \phi_\e \nabla u_\e \|_{L^{q'}} \le C\|f\|_{L^p(\Omega_\e)}.
\end{equation*}
Note that $\frac{1}{q'} = \frac{1}{p}-\frac{1}{d}$. Hence $q' = p^*$, which ends the proof.
\end{proof}

\section*{Appendix}
\renewcommand{\theequation}{A.\arabic{equation}}
\setcounter{equation}{0}
\renewcommand{\thetheorem}{A.\arabic{theorem}}
\setcounter{theorem}{0}

\noindent\textbf{Smoothing operators.}
Let $\alpha>0$ be fixed. Let $0\le \zeta \in C_0^\infty(B_{\alpha}(0))$ and $\int_{B_{\alpha}(0)} \zeta = 1$ and define the standard smoothing operator by
\begin{equation*}
    \mathscr{K}_\e f(x) = \int_{B_{\alpha \e}(x)} \e^{-d} \zeta(\frac{x-y}{\e})f(y) dy = \int_{B_{\alpha \e}(0)} \e^{-d} \zeta(\frac{y}{\e})f(x-y) dy.
\end{equation*}
Let $\Omega$ be a bounded domain in $\R^d$ and $\Omega(\e) = \{ x\in \Omega: \dist(x,\partial \Omega) < \alpha \e \}$. Note that if $f \in L^1(\Omega)$, then $\mathscr{K}_\e f$ is well-defined in $\Omega \setminus \Omega(\e)$.
Some properties of the smoothing operator are listed below, whose proofs may be found in \cite[Chapter 3.1]{Shen18}.

\begin{lemma}\label{lem.smoothing}
    Let $1\le p<\infty$. Assume $g\in L^p_{\rm per}(Y)$ and $f\in L^p(\Omega)$. Then
    \begin{equation*}
        \| g(x/\e) \mathscr{K}_\e f \|_{L^p(\Omega \setminus \Omega(\e))} \le C \|g \|_{L^p(Y)} \| f\|_{L^p(\Omega)},
    \end{equation*}
    where $C$ depends only on $p$ and $\alpha$.
\end{lemma}

\begin{lemma}\label{lem.smoothing.error}
    Let $1\le p<\infty$ and $\Omega$ be a bounded Lipschitz domain in $\R^d$. Assume $f\in W^{1,p}(\Omega)$. Then
    \begin{equation*}
        \| \mathscr{K}_\e f - f \|_{L^p(\Omega \setminus \Omega(\e))} \le C\e  \| \nabla f\|_{L^p(\Omega)},
    \end{equation*}
    where $C$ depends only on $p$ and $\alpha$.
\end{lemma}

\begin{lemma}\label{lem.bdrysmoothing}
    Let $\Omega$ be a bounded Lipschitz domain in $\R^d$ and $q = \frac{2d}{d+1}<2$. Assume $g\in L^2_{\rm per}(Y)$ and $f\in W^{1,q}(\Omega)$. Then
    \begin{equation*}
        \int_{\Omega(2t) \setminus \Omega(t)} |g(x/\e)|^2 |\mathscr{K}_\e f|^2 \le Ct \|g \|_{L^2(Y)}^2 \| f\|_{W^{1,q}(\Omega)}^2,
    \end{equation*}
    where $t\ge \e$ and $C$ depends only on $\Omega$ and $\alpha$.
\end{lemma}

The following is a related lemma without smoothing.
\begin{lemma}\label{lem.layerL2}
    Under the same assumption as Lemma \ref{lem.bdrysmoothing}, we have
    \begin{equation*}
        \|f \|_{L^2(\Omega(t))} \le Ct^\frac12 \| f\|_{W^{1,q}(\Omega)},
    \end{equation*}
    for any $t>0$.
\end{lemma}

\noindent\textbf{Nontangential maximal functions.}
We consider the solvability and regularity of the Dirichlet problem in a Lipschitz domain:
\begin{equation}\label{eq.Dp problem}
\left\{
\begin{aligned}
    -\Delta u & = 0 &\quad& \text{in } \Omega,\\
    u & = f &\quad & \text{on } \partial \Omega.
\end{aligned}
    \right.
\end{equation}
Assume $f\in L^2(\partial \Omega)$. Define the nontangential maximal function $u^*(Q)$ with $Q\in \partial \Omega$ by
\begin{equation*}
    u^*(Q) = \sup\{ |u(x)|: x\in \Omega \text{ and }  \dist(x,\partial \Omega) > \beta |x-Q| \},
\end{equation*}
where $\beta$ is a fixed constant chosen according to the Lipschitz character of $\partial \Omega$. It has been proved (see \cite{Da79}) that \eqref{eq.Dp problem} is solvable with $f\in L^2(\partial \Omega)$ and 
\begin{equation*}
    \| u^* \|_{L^2(\partial \Omega)} \le C\| f\|_{L^2(\partial \Omega)}.
\end{equation*}

Similarly, we may define
\begin{equation*}
    (\nabla u)^*(Q) = \sup\{ |\nabla u(x)|: x\in \Omega \text{ and }  \dist(x,\partial \Omega) > \beta |x-Q| \}.
\end{equation*}
Then if, in addition,  $f\in H^1(\partial \Omega)$ and $\partial \Omega$ is connected, it was proved (see \cite{JK}) that
\begin{equation}\label{est.L2reg}
    \| (\nabla u)^* \|_{L^2(\partial \Omega)} \le C\| \nabla_{\rm tan} f \|_{L^2(\partial \Omega)}.
\end{equation}
As a simple corollary, we have
\begin{lemma}\label{lem.nontangential.L2}
    Let $f\in H^1(\partial \Omega)$ and $u$ be a solution of \eqref{eq.Dp problem}. Assume $\partial \Omega$ is connected.
    Then
    \begin{equation*}
        \| \nabla u \|_{L^2(\Omega(t))} \le Ct^\frac{1}{2} \| \nabla_{\rm tan} f \|_{L^2(\partial \Omega)},
    \end{equation*}
    for any $t>0$.
\end{lemma}

\begin{proof}[Sketch of the proof of Lemma \ref{lem.u0inLip}]
    Let $v \in H^2(\R^d)$ be such that $\cL_0( v) = F \mathbbm{1}_{\Omega}$ in $\R^d$ and $\| v \|_{H^2(\R^d)} \le C\| F\|_{L^2(\Omega)}$. By the trace theorem,
    \begin{equation*}
        \| v \|_{H^1(\partial \Omega)} \le C\| F\|_{L^2(\Omega)},
    \end{equation*}
    and
    \begin{equation}\label{est.DvF}
        \int_{\Omega(\e)} |\nabla v|^2 \le C\e \| v \|_{H^2(\R^d)}^2 \le C\e \| F\|_{L^2(\Omega)}^2.
    \end{equation}
    Let $w = u_0 - v$, then $w$ satisfies
    \begin{equation*}
        \cL_0 (w) = 0 \quad \text{in } \Omega \quad \text{and} \quad w = g-v \quad \text{on } \partial \Omega.
    \end{equation*}
    By Lemma \ref{lem.nontangential.L2}, we have
    \begin{equation}\label{est.Dw*}
        \| (\nabla w)^* \|_{L^2(\partial \Omega)} \le C\| g - v\|_{H^1(\partial \Omega)} \le C\| g\|_{H^1(\partial \Omega)} + C\| F\|_{L^2(\Omega)}.
    \end{equation}
    It follows that
    \begin{equation*}
        \int_{\Omega(\e)} |\nabla w|^2 \le C\e \int_{\partial \Omega} |(\nabla w)^*|^2 \le C\e \big( \| g\|_{H^1(\partial \Omega)}^2 + \| F\|_{L^2(\Omega)}^2 \big).
    \end{equation*}
    This proves \eqref{est.u0layer-1} in view of \eqref{est.DvF}.

    Let $x\in \Omega$ and $\delta(x) = \dist(x,\partial \Omega)$. Then by the interior estimate,
    \begin{equation*}
        |\nabla^2 w(x)| \le \frac{C}{\delta(x)} \bigg( \fint_{B_{\delta(x)/2}} |\nabla w|^2 \bigg)^{1/2}
    \end{equation*}
    It follows that
    \begin{equation*}
        \int_{\{ \delta(x) = t \}} |\nabla^2 w|^2 \le Ct^{-3} \int_{t/2<\delta(y)<2t} |\nabla w|^2 dy \le Ct^{-2} \int_{\partial \Omega} |(\nabla w)^*|^2
    \end{equation*}
    Consequently, the co-area formula leads to
    \begin{equation*}
        \int_{\Omega\setminus \Omega(c\e)} |\nabla^2 w|^2 \le \int_{c\e}^{\text{diam}(\Omega)} Ct^{-2} dt\int_{\partial \Omega} |(\nabla w)^*|^2 \le C\e^{-1} \int_{\partial \Omega} |(\nabla w)^*|^2.
    \end{equation*}
    This estimate together with the $H^2$ estimate of $v$ and \eqref{est.Dw*} gives \eqref{est.u0layer-2}.
\end{proof}

\bibliography{Mybib}
\bibliographystyle{plain}
\medskip



\end{document}